\pdfoutput=1
\documentclass[reqno]{amsart}

\usepackage{enumerate}
\usepackage{centernot}
\usepackage{mathtools}
\usepackage{stmaryrd}
\usepackage{float}
\usepackage{amssymb}
\usepackage{graphicx}
 \usepackage{amsfonts}
\usepackage{microtype} 
\usepackage{tikz-cd}
\usepackage[symbol]{footmisc}
\usepackage[english]{babel}
\usepackage[pagewise]{lineno}
\usepackage{blindtext}
\usepackage{kantlipsum}
\usepackage{array}
\usepackage{mathrsfs}

\newtheorem{theorem}{Theorem}[section]
\newtheorem{thmx}{Theorem}

\newtheorem{proposition}[theorem]{Proposition}
\newtheorem{lemma}[theorem]{Lemma}
\newtheorem{corollary}[theorem]{Corollary}
\theoremstyle{definition}
\newtheorem{definition}[theorem]{Definition}
\newtheorem{remark}{Remark}[section]
\newtheorem{question}{Question}

\usepackage{hyperref}
\hypersetup{
colorlinks   = true, 
urlcolor     = oxfordblue, 
linkcolor    = oxfordblue, 
citecolor   = phthalogreen 
}

\usepackage[style=alphabetic, url = false, eprint=false, maxnames=10]{biblatex}
\numberwithin{equation}{section}
\definecolor{oxfordblue}{rgb}{0.0, 0.13, 0.28}
\definecolor{phthalogreen}{rgb}{0.07, 0.21, 0.14}

\newcommand\numberthis{\addtocounter{equation}{1}\tag{\theequation}}

\makeatletter
\newsavebox\myboxA
\newsavebox\myboxB
\newlength\mylenA
\newcommand*\xoverline[2][0.70]{%
	\sbox{\myboxA}{$\m@th#2$}%
	\setbox\myboxB\null
	\ht\myboxB=\ht\myboxA%
	\dp\myboxB=\dp\myboxA%
	\wd\myboxB=#1\wd\myboxA
	\sbox\myboxB{$\m@th\overline{\copy\myboxB}$}
	\setlength\mylenA{\the\wd\myboxA}
	\addtolength\mylenA{-\the\wd\myboxB}%
	\ifdim\wd\myboxB<\wd\myboxA%
	\rlap{\hskip 0.8\mylenA\usebox\myboxB}{\usebox\myboxA}%
	\else
	\hskip -0.5\mylenA\rlap{\usebox\myboxA}{\hskip 0.5\mylenA\usebox\myboxB}%
	\fi}
\makeatother

\newcommand{\bmat}[1]{\begin{bmatrix} #1 \end{bmatrix}}

\usepackage{hyperref}
\usepackage{cleveref}
\usepackage{theoremref}
\usepackage{lmodern}

\newcommand\R{\mathbb{R}}
\newcommand\Z{\mathbb{Z}}

\let\geq\geqslant

\DeclareMathOperator*{\Span}{span}
\DeclareMathOperator*{\supp}{supp}

\newcommand\interint[2]{\left\llbracket 1, 5\right\rrbracket}

{\end{tcolorbox}}

\newcommand{\mC}{\ensuremath{\mathcal{C}}}

\newcommand{\mU}{\ensuremath{\mathcal{U}}}

\newcommand{\mT}{\ensuremath{\mathcal{T}}}
\newcommand{\mS}{\ensuremath{\mathcal{S}}}

\newcommand{\mE}{\ensuremath{\mathcal{E}}}
\newcommand{\mG}{\ensuremath{\mathcal{G}}}
\newcommand{\mW}{\ensuremath{\mathcal{W}}}
\newcommand{\mY}{\ensuremath{\mathcal{Y}}}
\newcommand{\mN}{\ensuremath{\mathcal{N}}}

\newcommand{\mV}{\ensuremath{\mathcal{V}}}

\newcommand{\bL}{\ensuremath{\mathbb{L}}}

\newcommand{\mM}{\ensuremath{\mathcal{M}}}
\newcommand{\mP}{\ensuremath{\mathcal{P}}}
\newcommand{\mK}{\ensuremath{\mathcal{K}}}
\newcommand{\mD}{\ensuremath{\mathcal{D}}}
\newcommand{\mA}{\ensuremath{\mathcal{A}}}
\newcommand{\mO}{\ensuremath{\mathcal{O}}}
\newcommand{\mI}{\ensuremath{\mathcal{I}}}

\newcommand{\Tm}{\ensuremath{\mathbb{T}}}

\newcommand{\qd}{\mathrm{quad}}
\newcommand{\rf}{\mathrm{rf}}

\renewcommand{\c}{\ensuremath{\mathfrak{c}}} 

\DeclareMathOperator{\Ad}{\mathrm{Ad} }
\newcommand{\cO}{\mathcal{O}}

\newcommand{\cG}{\mathcal{G}}
\newcommand{\cV}{\mathcal{V}}
\newcommand{\bn}{\mathbf{n}}
\newcommand{\cK}{\mathcal{K}}
\newcommand{\st}{:\,}
\newcommand{\C}{\mathbb{C}}
\newcommand{\diag}{\mathrm{diag}}
\newcommand{\mane}{Ma\~n\'e\ }

\newcommand{\cA}{\mathcal{A}}
\newcommand{\bJ}{\mathbb{J}}
\newcommand{\cU}{\mathcal{U}}
\newcommand{\cD}{\mathcal{D}}
\newcommand{\cM}{\mathcal{M}}
\newcommand{\cE}{\mathcal{E}}

\newcommand{\cH}{\mathcal{H}}

\title[Sub-Riemannian bumpy metric theorem]
{Bumpy metric theorem for co-rank 1 sub-Riemannian and reversible sub-Finsler metrics}
\author{Shahriar aslani}
\address{Department of Mathematics \\ University of Toronto}
\email{saslani@math.toronto.edu}
\author{Ke Zhang}
\address{Department of Mathematics \\ University of Toronto}
\email{kzhang@math.toronto.edu}
\begin{document}
\maketitle
\begin{abstract}
		We prove that for a generic sub-Riemannian and reversible sub-Finsler metrics defined on a fixed co-rank $1$ distribution, all strictly normal periodic orbits are non-degenerate. 
\end{abstract}
\tableofcontents

\section{Introduction}

A Riemannian metric on a manifold $M$ is called \emph{bumpy} if all of its closed geodesics are \emph{non-degenerate}, that is the linearized Poincar\'e map does not admit any roots of unity as an eigenvalue. The \emph{bumpy metric theorem} states that a generic Riemannian metric is bumpy. This result has a wide range of applications because the bumpy metrics play the role of Morse functions in the space of metrics. 

The Riemannian bumpy metric theorem was stated by Abraham \cite{Abrahambumpy}, and proved after series of contributions by Abraham \cite{Abraham67, Abrahambumpy}, Klingenberg and Takens \cite{KT72}, and Anosov \cite{Anosov_1983}. The Hamiltonian analog of the bumpy metric theorem is proven by Robinson (\cite{Rob70}, \cite{Rob70a}). 

A weaker sense of genericity known as \emph{Ma\~n\'e-genericity} allow only a generic potential perturbation $U \in C^\infty(M)$ to the Hamiltonian. Ma\~n\'e-generic properties of periodic orbits have been studied for fiberwise convex Hamiltonians \cite{Cont10,RR12,LRR16, AB21}, and only recently for fiberwise non-convex Hamiltonians \cite{aslani2022bumpy}.

In the sub-Riemannian case, the Lagrangian is defined on a positive co-rank distribution. As such, there is no analog of the Euler-Lagrangian flow, and the proper generalization of the geodesic flow is the flow of the dual Hamiltonian system. The bumpy metric theorem in this case concerns all Hamiltonian periodic orbits. We call the projection of a Hamiltonian orbit a \emph{normal orbit}.

A main feature of sub-Riemannian geometry is the existence of \emph{singular curves}, which has no analog in Riemannian geometry. We say a normal orbit is \emph{strictly normal} if it is not singular. Our main result is that for co-rank $1$ distributions,  bumpy metric theorem holds for the \emph{strictly normal} orbits for sub-Riemannian metric and generalizations. The requirement for strict normal orbits is necessary, see Remark \ref{rmk:regular} below. We consider a generic metric defined on a fixed distribution, which places our concept of genericity in between the Hamiltonian sense and the \mane generic case. 

We proceed with formal definitions and the formulation of the main theorem.
\begin{definition}
	Let $M$ be a $(d + 1)$-dimensional manifold, $\mD$ be a smooth $m$-dimensional sub-bundle to the tangent bundle $TM$. A \emph{sub-bundle Lagrangian} on $\mD$ (we will write $\mD$-Lagrangian for short) is a smooth fiberwise strictly convex function on $\cD$, i.e. $\partial^2_{vv} L(q, v)$ is strictly positive definite for all $(q, v) \in \mD \subset TM$.
\end{definition}
Two particular classes of sub-bundle Lagrangians are:
\begin{itemize}
		\item (Sub-Riemannian Lagrangian) $L(q, v)$ is a positive definite quadratic form in $v$, that is, there exists a smooth inner product $\langle \cdot, \cdot \rangle_q$ on the $q$-fiber $\mD_q$, such that $L(q, v) = \frac12 \langle v, v\rangle_q$.
		\item (Sub-Finsler Lagrangian) $L(q, v)$ is positively homogeneous of degree $2$, that i.e. $L(q, \lambda v) = \lambda^2 L(q, v)$ for all $\lambda > 0$.
\end{itemize}
We refer to \cite{Agrachevbook, Rif14, Montgomerybook} for a background in sub-Riemannian geometry, and \cite{AS04} for general control theory. In sub-Riemannian geometry, one typically impose that $\cD$ is totally nonholonomic ($\cD$ satisfies the H\"ormander condition), but we do not make this assumption, see the explanations in Remark \ref{rmk:regular}.

Define the extended Lagrangian
\begin{equation}  \label{eq_1.2}
		\tilde{L}(q, v) = 
		\begin{cases}
				L(q, v) & v \in \mD_q \\
				\infty &  v \notin \mD_q
		\end{cases}.
\end{equation}
The dual Hamiltonian $H: T^*M \to \R$ with respect to to $\tilde{L}$ is
$$
\begin{aligned}
		H(q, p) & = \sup \big\{p \cdot v - \tilde{L}(q, v) \mid \, v \in T_q M\big\} \\
						& = \sup \big\{ p \cdot v - L(q, v)\mid \, v \in \mD_q\big\}.
\end{aligned}
$$
Denote by $\mD_q^{\perp} \subset T_q^*M$ the annihilator of $\mD_q \subset T_qM$ defined as 
\begin{equation}
		\mD_q^\perp:=\{p \in T^*_qM  \mid p(v)=0, \quad \text{for all } \hspace{1.2mm} v \in T_qM\}.
\end{equation}
we say $H:T^*M \to \mathbb{R}$ is a \emph{sub-bundle Hamiltonian} on $\cD$ ($\cD$-Hamiltonian for short) if
\begin{equation}  \label{eq:H-flat}
		H(q, p + P) = H(q, p) , \quad \forall P\in \mD^\perp_q,
\end{equation}      
and that $H$ defined on the quotient space $T^*M/\cD^\perp$ is fiberwise strictly convex. A function $L:\mD \to \mathbb{R}$ is a $\cD$-Lagrangian if and only if the Legendre-Fenchel dual of $\tilde{L}$ given in (\ref{eq_1.2}) is a $\cD$-Hamiltonian.

In the sub-bundle setting, the relation between the Lagrangian and Hamiltonian formalism is more complicated due to the existence of \emph{singular curves}.

A curve $Q: [0, T] \to M$ is called \emph{horizontal} to $\cD$ if $Q \in H^1([0, T])$ and $\dot{Q}(t) \in \cD_{Q(t)}$ for a.e. $t$. The space of horizontal curves can be described using the language of control theory. Suppose $f^1, \cdots, f^m$ are smooth vector fields (called a \emph{local frame}) on an open neighborhood $O \subset M$ such that $\cD = \Span\{f^1, \cdots, f^m\}$. We consider a \emph{control} $\c=(\c_1,\c_2, \hdots ,\c_k) \in L^2\big([0,T];\mathbb{R}^d\big)$ and a solution to the ODE 
\begin{equation}\label{eq_CP1}
		\dot{Q}(t)=\sum_{i=1}^d
		\c_i(t)f^i_{Q(t)}\big(Q(t)\big), \quad Q(0)=\bar{Q}.
\end{equation}
If a control $\c: [0, T] \to \R^2$ corresponds to a horizontal curve $Q_\c: [0, T] \to M$ short enough to fit into a local frame, then we can define the \emph{end-point mapping} $\cE$ as $\c \mapsto Q_\c(T)$.

\begin{definition}
		 The curve $Q_\c$ is called $\cD$-\emph{regular} if $d\cE(\c): L^2([0, T]) \to T_{Q_\c(T)}M$ is a surjection, otherwise $Q_\c$ is called $\cD$-\emph{singular}. (This definition is independent of the choice of local frames, see Appendix D of \cite{Montgomerybook}).

In general, we say a curve is $\cD$-singular if it's singular on every local frame.
\end{definition}

If $L$ represent a sub-Riemannian metric, then 
\begin{itemize}
		\item If $Q$ is a $\cD$-regular geodesic, then there exists an orbit $\big(Q(t), P(t)\big)$ of the Hamiltonian flow of $H$ which projects to it. In this case $Q$ is called a \emph{normal} orbit. 
		\item $\cD$-singular geodesics exists, it may or may not be the projection of a normal Hamiltonian orbit (the first example of an \emph{abnormal singular geodesic} is given in \cite{Mon94}).
\end{itemize}
More generally, a curve $Q:[0, T] \to M$ is called a \emph{normal orbit} of the $\cD$-Lagrangian $L$ if it is the projection of its corresponding $\cD$-Hamiltonian orbit. We call $\cD$-regular normal orbit \emph{strictly normal}, noting that $\cD$-singular normal orbits also exists. Since Hamiltonian orbits are by definition normal, we drop this description in the Hamiltonian setting.

To state our main theorem, denote by
\[
		\begin{aligned}
				\cH_{\cD} & = \{K = K(q, p) \in C^\infty(T^*M) \mid \quad K \text{ is a } \cD\text{-Hamiltonian}\}, \\
				\cK_{\cD}^{\qd} & = \{K \in \cH_D \mid  \quad K \text{ is fiber-wise quadratic}\}, \\
				\cK_{\cD}^{\beta} & = \{K \in \cH_D \mid  \quad K(q, \lambda p) = \lambda^\beta K(q, p), \quad \forall p \in T^*_q M, \, \lambda > 0\}, \\
				\cK_{\cD}^{\rf} & = \{ K \in \cK_{\cD}^2 \mid \quad K(q, -p) = K(q, p) \quad \forall p \in T_q^*M\}. 
		\end{aligned}
\]
the Hamiltonians in the latter three spaces are called quadratic, $\beta$-homogeneous, and reversible Finsler, respectively. Analogous to classical mechanics, we call the following system
\[
		H(q, p) = K(p) + U(q)
\]
a \emph{generalized classical system}, 
where $K$ is taken from $\cK_\cD^{\qd}$, $\cK_\cD^\beta$, or $\cK_\cD^{\rf}$, and $U$ is a bounded smooth function, i.e. $U \in C_b^\infty(M)$. Note that every generalized classical system is a $\cD$-Hamiltonian. 

An energy level $\{K(q, p) + U(q) = k\}$ of a generalized classical system is \emph{supercritical} if 
\[
		k - \sup_q U(q) > 0.   
\]
For generalized classical system, supercriticality of $k$ is equivalent to the canonical one-form $pdq$ being contact on the energy level in a uniform sense:
\[
		\inf_{H(q, p) = k} \partial_p H(q, p) \cdot p > 0.
\]
In particular, $\partial_p H(q, p)$ never vanishes. We call the tuple $(K, U, k)$ \emph{admissible} if $k$ is supercritical for $K + U$. 

\begin{thmx}[main theorem]\label{main_thm}
		Let $\mD$ be a co-rank 1 distribution on $M$, $k>0$ be given, and $\cK_\cD^\star$ denote either $\cK_\cD^{\qd}$ or $\cK_\cD^{\rf}$. Then there exists a $G_\delta$ dense subset $\mG$ of 
		\[
				\{(K, U) \in \cK_\cD^{\star} \times C_b^\infty(M) \mid  \hspace{1.2mm} (K, U, k) \text{ is admissible}\},
		\]
		such that for all $(K, U) \in \mG$, every $\cD$-regular closed orbit of the Hamiltonian flow $H = K + U$ with energy $k$ is non-degenerate.
\end{thmx}

\begin{remark}\label{rmk:regular}
		By stating our main theorem in terms of regular orbits, we can separate the effect of the Hamiltonian and the distribution on the dynamics. 

		First of all, this allows us to prove a theorem for a \emph{fixed} distribution $\cD$. This is a stronger sense of genericity compared to a perturbation to the Hamiltonian that also alters the distribution.

		Secondly, if the distribution $\cD$ is integrable (therefore fails the total nonholonomic condition), the bumpy metric theorem can fail. For example, if $M = \Tm^3$ and $\cD = \Span\{\partial_x, \partial_y\}$, then all linearized Poincar\'e maps are degenerate. However, this example does not contradict Theorem \ref{main_thm}, since all horizontal curves for this distribution are singular, and our theorem is vacuous in this case.

		Finally, some distributions admit no singular curves. Examples include all contact distributions on odd dimensional spaces, and more generally, all \emph{fat} distributions (see for example section 5.6 of \cite{Montgomerybook}). In this case our theorem imply the classical bumpy metric theorem for all periodic orbits on admissible energy levels.
\end{remark}

\begin{remark}\label{rmk:reversible}
		Observe that in both settings of our main theorem, the Hamiltonian flow is \emph{reversible}, namely, the map $(q, p) \mapsto (q, -p)$ conjugate the flow to its time-reversal.

		Given a periodic orbit $\theta: \R/(s\Z) \to T^*M$, we say that $t \in \R/(s\Z)$ is neat if $\pi \circ \theta(t)$ is not a self-intersection of the projected orbit. It has been observed by \cite{aslani2022bumpy}, and \cite{Ber24} that perturbation by a potential function faces additional difficulty at non-neat times. Moreover, orbits without neat times do exist as \emph{round trip orbits} (see \cite{Ber24}). For our proof, we also need the neat time to be in a $\cD$-regular interval. 

		It turns out this difficulty can be avoided if we consider reversible systems on a supercritical energy level. One of the reasons is that no round trip orbits exists on these levels.
\end{remark}

The \emph{Maupertuis principle} relates an admissible generalized classical system to a metric. More precisely, if $(K, U, k)$ is admissible, with $K \in K_\cD^{\star}$, $\star = \qd, \rf$, then
\[
		\tilde{K}(q, p) = \frac{K}{k - U} \in \cK_\cD^{\star}
\]
satisfies the property that the Hamiltonian flow of $K + U$ on energy level $k$ is a time change of the Hamiltonian flow of $\tilde{K}$ on energy level $1$. As such $K + U$ is bumpy on energy $k$ if and only if $\tilde{K}$ is bumpy on energy $1$. 

Using this observation, we obtain the following corollary of our main theorem.

\begin{thmx}[Application to sub-Riemannian and sub-Finsler metrics]\label{sub-Riemannian}
		Let $\cD$ be a co-rank $1$ distribution on $M$. Then there exists a dense $G_\delta$ set $\cG \subset \cK_\cD^\star$, where $\star = \{\qd, \rf\}$, such that for every $K \in \cG$, every $\cD$-regular normal closed orbit is non-degenerate.
\end{thmx}

To prove the main result, we prove a \textit{local Ma\~n\'e perturbation theorem} for the linearized Poincar\'e maps which are valued in the space of symplectic linear maps $Sp(2d, \R)$; refer to Theorem \ref{pert_thm} below.  

Similar results have been achieved for neat periodic orbits of fiberwise non-degenerate Hamiltonians \cite{RR12, AB21, aslani2022bumpy}. It is much more challenging to derive an analogous result for $\cD$-Hamiltonians due to the degeneracy of $H$. 

We define a local and intrinsic version of the $\cD$-regularity property in the special case of co-rank $1$ distributions.

\begin{definition}\label{def:d-reg}
		Let $Q: [0, T] \to M$ be a smooth horizontal curve and $\eta$ is a non-vanishing one form taking value in the one-dimensional sub-bundle $\cD^\perp$. We say that $t \in (0, T)$ is a $\cD$-\emph{regular time} if the linear form
		\[
				(d_{Q(t)}\eta)(\dot{Q}(t), \cdot) : \, \cD_{Q(t)} \to \R 
		\]
		is nonzero.
\end{definition}
This definition is independent of the choice of the one-form $\eta$, since for all $v \in \cD_{Q(t)}$ and a non-vanishing $f \in C^\infty(M)$,
\[
		d(f\eta)\big(\dot{Q}(t), v\big) = (d f) \wedge \eta \big(\dot{Q}(t), v\big) + f (d\eta)\big(\dot{Q}(t), v\big) =  f (d\eta)\big(\dot{Q}(t), v\big).
\]
The following lemma follows from Theorem 5.3 and Lemma 5.5 of \cite{Montgomerybook}. 
\begin{lemma}\label{lem:regular-point}
		For a co-rank $1$ distribution $\cD$, a smooth curve $Q: [0, T] \to M$ is $\cD$-singular if and only if there are no regular times on $Q$, i.e. $(d\eta)_{Q(t)}\big(\dot{Q}(t), v\big) = 0$ for all $v \in \cD_{Q(t)}$.
\end{lemma}

Consider a generalized classical Hamiltonian $H = K + U,$ where $K \in \mK_\mD^\star$, with $\star$ being either $\qd$ or $\rf$. Assume that $(K, U, k)$ is admissible and $\theta: [0, \delta] \to T^*M$ is an orbit segment in $\{H = k\}$. Denote
\begin{equation}  \label{eq:admissible-K}
		\mK^\star(\mD, K, \theta) = \big\{T \in \mK_\mD^\star \mid  d T \big(\theta(t)\big) = d K\big(\theta(t)\big), \hspace{1.3mm} t \in [0, \delta]\big\},
\end{equation}
and 
\begin{equation}  \label{eq:cU}
\begin{aligned}
		\mU(\theta) = \big\{u \in C^\infty_b(M) \mid u(\pi \circ \theta(t)) =0, \hspace{2mm}  du(\pi \circ \theta(t))=0, \,  
		\quad t \in [0, \delta]\big\}.
\end{aligned}
\end{equation}
Note that all $H = T + U + u$ with $T \in \cK^\star(\cD, K, \theta)$ and $u \in \cU(\theta)$ admits $\theta$ as an orbit on the energy level $k$.

Let $\Sigma_0$ and $\Sigma_\delta$ be two transversal sections to $\theta(t)$ at $t = 0$ and $t = \delta$ respectively. Denote
\[
		V_0 =   T_{\theta(0)}\left( \Sigma_0 \cap \{H = k\} \right), \quad
		V_\delta = T_{\theta(\delta)} \left( \Sigma_0 \cap \{H = k\}\right),
\]
and the linearized Poincar\'{e} map of the Hamiltonian flow of $H$ at $\theta(0)$ as
$$
dP_{\Sigma_0, \Sigma_\delta}(\theta,H): V_0 \to V_\delta.
$$
Let $\mV$ be a neighborhood of $K$ in $\mK(\mD, K, \theta),$ then the mapping 
\begin{equation}  \label{eq:linear-poincare}
		\mU(\theta) \ni u \mapsto  dP_{\Sigma_0, \Sigma_\delta}(\theta, T+U+u)
\end{equation}
is well defined for all $T \in \mV$ and $u \in \mU(U, \theta)$. The linearized Poincar\'e map $dP_{\Sigma_0, \Sigma_\delta}$ is symplectic with respect to the restriction of the standard symplectic form. In the Daboux coordinates, this takes value in $Sp(2d, \R)$.

\begin{thmx}\label{pert_thm}(Perturbation theorem for restricted transition maps)
		Let $M$ be a smooth manifold and $\mD \subset TM$ be a distribution of co-rank 1. Assume that $\theta:[0, \delta] \to T^*M$ is an orbit of a generalized classical Hamiltonian $H = K + U$ where $K \in \mK_\mD^\star$ with $\star \in \{\qd , \rf\}$, and let $\Sigma_0$, $\Sigma_\delta$ be two transversal sections to $\theta$ at $t = 0$ and $t = \delta$. Assume that $\pi \circ \theta$ has non-zero velocity at $t=0$, and $t = 0$ is a $\cD$-regular time. 

		Then exists an open and dense subset $\mV$ of $\cK^\star(\cD, K, \theta)$ such that for every $T \in \mV$, and any neighbhorhood of $N$ of the set $\{\pi \circ \theta(t) \mid t \in (0, \delta)\}$, the mapping
		$$
		\mU(\theta) \cap \{\supp u \in N\} \ni u \mapsto dP_{\Sigma_0, \Sigma_\delta}(\theta, T+U+u),
		$$
		is a submersion at $u = 0$. 
\end{thmx}
\begin{remark}
		If the conclusion of the theorem holds for some transversal sections $\Sigma_0$, $\Sigma_\delta$, then it holds for arbitrary choice of transversal sections.
\end{remark}

\subsection*{Dicussions and questions}

One of the central questions in sub-Riemannian geometry is whether singular curves are rare. This is known as the Sard's Conjecture (see \cite{SFPR22} for a recent progress in this problem). In our context, it's natural to ask whether singular normal orbits are rare. In particular, we raise the following question:
\begin{question}
	For a totally nonholonomic distribution $\cD$ and a generic sub-Riemannian metric, are all periodic orbits $\cD$-regular?
\end{question}
If the answer is yes, then we can remove the regularity requirement from our main theorem in the nonholonomic case.

In Theorem \ref{pert_thm}, we need a preliminary perturbation $T$ to the kinetic energy in order to prove that the mapping $u \mapsto dP$ is an submersion. Can this perturbation be avoided if $\cD$ is totally nonholonomic? This raises the following question:  

\begin{question}
		For a totally nonholonomic distribution $\cD$ and a generic sub-Riemannian metric, do Theorem \ref{pert_thm} hold with only perturbation in $U$? In other words, does the bumpy metric theorem hold in the sense of Ma\~n\'e?
\end{question}

\subsection*{Outline of the proof}

\begin{itemize}
		\item  In Section \ref{sec:local-pert}, we prove a normal form theorem. Using the normal form, we formulate Theorem \ref{thm:controllable} concerning a matrix control problem. We then prove Theorem \ref{pert_thm} assuming Theorem \ref{thm:controllable}.
		\item In Section \ref{sec:param-trans}, we use Theorem \ref{pert_thm} and the parametric transverality method to prove Theorem \ref{main_thm} and derive Theorem \ref{sub-Riemannian} as a corollary.
		\item In Section \ref{sec:mane-control}, we prove Theorem \ref{thm:controllable}.
\end{itemize}

\section{Proof of the local perturbation theorem via Ma\~n\'e control problem}
\label{sec:local-pert}

In this section, we reduce the proof of Theorem \ref{pert_thm} to a controllability problem (Theorem \ref{thm:controllable}). 

The plan of the section is as follows.
\begin{itemize}
		\item In Section \ref{sec:normal-form}, we prove a normal form theorem using fiber-preserving coordinate changes. Although this normal is related to the ones in \cite{RR12} and \cite{AB21}, it has a different block form compared to \cite{RR12} and \cite{AB21}. The normal form theorem applies to all  $\cD$-Hamiltonians.
		\item In Section \ref{sec:lin-trans}, we use the normal form to reduce the linearized Poincar\'e map to the solution of a control problem. The result of this section applies to distributions of all ranks.
		\item In Section \ref{sec:mane-controllable}, we restrict to the co-rank $1$ case, and formulate Theorem \ref{thm:controllable}. The proof is deferred to the end of the paper.
		\item In Section \ref{sec:proof-pert-thm}, we prove Theorem \ref{pert_thm} assuming Theorem \ref{thm:controllable}.
\end{itemize}

\subsection{Normal form along orbit segments of a generalized classical Hamiltonian} \label{sec:normal-form}
Let $\varphi$ be a diffeomorphism defined on an open subset $V \subset M.$ A mapping $\Psi_{\varphi}:T^*V \to T^*\big(\varphi(V)\big)$ defined as 
$$
\Psi_{\varphi}(q,p)=\big(\varphi(q),\big(d\varphi^{-1}(q)\big)^*p\big)
$$
is called a \textit{homogeneous symplectomorphism}. 

Note that Homogeneous symplectomorphisms are functorials i.e.
$$
    \Psi^{-1}_{\varphi}=\Psi_{\varphi^{-1}}, \quad \text{and} \quad \Psi_{\chi \circ \varphi}=\Psi_{\chi} \circ \Psi_{\varphi},
$$
where $\chi$ is a diffeomorphism defined on $\varphi(V).$ 

Let $g:V \to \mathbb{R}$ be a smooth function defined on a open set $V \subset M.$ The mapping 
$$\Psi^g(q,p)=\big(q,p+dg(q)\big)$$ 
Defines a \textit{vertical symplectomorphism} on $T^*V$ to itself. Note that as we have 
\begin{equation}
		\Psi^{g_1+g_2}=\Psi^{g_1} \circ \Psi^{g_2},
\end{equation}
the set of vertical symplectomorphisms on $T^*V$ forms a group. 

In a local coordinates we may express homogeneous and vertical symplectomorphisms in the following manner 
\begin{equation}\label{eq_hom-ver}
		\Psi_\varphi(q,p)=\big(\varphi(q),(D\varphi^{-1})^Tp\big), \quad \Psi^g(q,p)=\big(q,p+(Dg)^T(q)\big).  
\end{equation}

It is known that any symplectomorphism on $T^*M$ which preserves the fibers is a composition of a homogeneous (vertical) and a vertical (homogeneous) symplectomorphism. See Lemma 1.1.3 in \cite{Aslanithesis}. The next lemma claims the two types of sympletomorphisms commute in a sense.
\begin{lemma}\label{lemma_2.1}
		Let $\varphi$ be a diffeomorphism defined on an open set $V \subset M,$ and $g:V \to \mathbb{R}$ be a smooth function. Consider a homogeneous symplectomorphism 
		$$\Psi_\varphi(q,p)=\big(\varphi(q),(d\varphi^{-1})^* p\big),$$
		and a vertical symplectomorphism  
		$$\Psi_{g}(q,p)=\big(q,p+dg(q)\big).$$ 
		Then, $\Psi^g$ and $\Psi_\varphi$ commute in the following way 
		$$
		\Psi^g \circ \Psi_\varphi=\Psi_{\varphi} \circ \Psi^{g \circ \varphi}.
		$$
\end{lemma}
As a result, if $\Phi_1, \hdots, \Phi_k:T^*V \to T^*V$ is a collection of homogeneous and vertical symplectomorphisms, then there exists a smooth function $f:V\to \mathbb{R},$ and a diffeomorphism $\chi$ defined on $V$ such that
$$
\Phi_1 \circ \hdots \circ \Phi_k=\Psi_{\chi} \circ \Psi_f.
$$

\begin{proof}[Proof of Lemma \ref{lemma_2.1}]
In coordinates, we have
		\begin{align*}
				\big(\Psi^g \circ \Psi_\varphi\big)(q,p)&=\big(q,p+\big(Dg(q)\big)^T(q)\big) \circ \Psi_{\varphi} \\
						&= \big(\varphi(q), (D\varphi^{-1}(q))^Tp+(Dg(\varphi(q))^T\big) \\ 
						&= \big(\varphi(q), (D\varphi^{-1}(q))^T\big(p+(D\varphi(q))^T(Dg(\varphi(q))^T\big)\big) \\  \numberthis \label{comp_diff}
						&= \big(\varphi(q), (D\varphi^{-1}(q))^T\big(p+(D (g \circ \varphi ))^T\big)\big) \\
						&=\big(\Psi_\varphi \circ \Psi^{g \circ \varphi}\big)(q,p).
		\end{align*}
		Where we obtained (\ref{comp_diff})  using 
            $$(D (g \circ \varphi ))^T=(D\varphi(q))^T(Dg(\varphi(q))^T.$$
\end{proof}

The proof of the following lemma is straightforward by the definitions of $\mD$-Hamiltonian.
\begin{lemma} \label{lemma_2.2}
		Let $\mD$ be distribution over a smooth manifold $M.$ Let $V$ be an open subset of $M$. Suppose that $H:T^*V \to \mathbb{R}$ is a $\cD$-Hamiltonian. Let $\Psi_{\varphi}$ be a homogeneous symplectomorphism and $\Psi^g$ be a vertical symplectomorphism, both defined on $T^*V.$ Denoted by $\mD|_V$ the restriction of the bundle $\mD$ to $V$ we have the following  
		\begin{itemize}
				\item $H \circ \Psi^{-1}_{\varphi}$ is a $\varphi_* \cD$-Hamiltonian defined on $T^*\big(\varphi(V)\big)$.
				\item $H \circ \Psi^g$ is a $\cD$-Hamiltonian defined on $T^*V$.
		\end{itemize}
\end{lemma}

Let $\Psi_\varphi$ be a homogeneous symplectomorphism. For a generalized classical Hamiltonian $H=K+U,$ we have 
$$
(H \circ \Psi_\varphi)(q,p)=K\big(\varphi(q),(d\varphi^{-1})^*p\big)+U\big(\varphi(q)\big)=:\tilde{K}(q,p)+\tilde{U}.
$$
Therefore, homogeneous change of coordinates preserve a generalized classical Hamiltonian. At the other hand, for a vertical symplectomorphism $\Psi^g$ 
\begin{equation}\label{eq_2.1}
		(H \circ \Psi^g)(q,p)=K\big(q,p+dg(q)\big)+U(q)
\end{equation}
is no longer a generalized classical Hamiltonian. That is because $K\big(q,p+dg(q)\big)$ is not homogeneous. 

Due to Lemma \ref{lemma_2.1}, after composing a generalized classical system to a finite number of homogeneous and vertical change of coordinates, the new Hamiltonian is of the form given in (\ref{eq_2.1}).

\begin{proposition}\label{normalform}
		Let $\mD$ be a distribution of rank $m < \dim M = d+1$. Let $\underline{H}$ be a $\cD$-Hamiltonian, and $\underline{\theta}=\big(\underline{Q}(t),\underline{P}(t)\big):[0,T] \to T^*M$ be an orbit on the energy level $k$ satisfying $\dot {\underline{Q}}(0) \ne 0.$

    Then there exists $\delta>0,$ a neighborhood $V$ of $\underline{\theta}(0),$ a diffeomorphism $\varphi:V \to \tilde{V}:=\varphi(V) \subset \mathbb{R}^{d+1}$ with $\varphi\big(\underline{Q}(0)\big)=0,$ and a function $g:\tilde{V} \to \mathbb{R}$ such that, if we set 
		$$
		H:=\underline{H} \circ \Phi, \quad \Phi:=\Psi^{-1}_{\varphi} \circ (\Psi^g)^{-1}, \quad \text{where $\Psi_{\varphi}$ and $\Psi^g$ are defined as in (\ref{eq_hom-ver}),}
		$$
		then for all $q \in V$ 
		\begin{enumerate}[$(a)$]
				\item  $\theta(t):=\Phi\big(\underline{\theta}(t)\big)=(te_1,0), \quad t \in [0, \delta],$ 
				\item  $H(q,0)=k,$ 
				\item  $\partial_pH(q,0)=e_1,$ 
				\item  $\partial^2_{\hat{p}\hat{p}}H(0,0)=
						\begin{bmatrix}
								I_{d-m} & 0 \\
								0 & 0 
						\end{bmatrix}.$ 
		\end{enumerate}
\end{proposition}
\begin{proof}
		Throughout the proof, in each of the intermediate steps we assume that $\underline{H}$ is the Hamiltonian obtained from the previous step of normalization. The Hamiltonian obtained in the current step is denoted $H$. We use the notations $\underline{\theta}$ and $\theta$ in the same manner. 

		\textit{Step 1.} By assumption, $\dot{\underline{Q}}(t) = \partial_p\underline{H}\big(\underline{\theta}(0)\big)\ne 0$. Hence there exists $\delta>0$ and a diffeomorphism $\varphi_1$ defined on a neighborhood $V$ of $\underline{Q}(0)$ such that $\varphi_1\big(\underline{Q}(t)\big)=te_1,$ for all $t\in [-\delta,\delta].$ Let 
		$$
		\underline{H} \circ \Psi_{\varphi_1}^{-1}=:H=K+U,
		$$ 
		and 
		$$
		\Psi_{\varphi_1}(\underline{\theta}(t))=:\theta(t)=:\big(te_1,P(t)\big):[-\delta, \delta] \to T^*M.
		$$ 
		
		In further steps, by choosing a sufficiently small $\delta$ and pulling back to a local chart, we will assume that the projection of a transformed orbit segment lies in $V \subset \R^{d+1}$. 
  
		\textit{Step 2.} We let $\underline{\theta}(t)=\big(te_1, \underline{P}(t)\big)$. Using the method of characteristics (see \cite{Evans10}, Theorem 2, Section 3.2.4), there exists a neighborhood $\hat{V} \subset \{0\} \times \mathbb{R}^d$ of $\hat{q} = 0$ and $\delta > 0$, such that the following boundary value problem of the Hamilton-Jacobi equation has a smooth solution
		\[
				\begin{aligned}
											& \underline{H}\big(q, dg(q)\big) - k = 0, & q_1 \in (-\delta, \delta), 	\hspace{1.2 mm} \hat{q} \in \hat{V}, \\
											& g(0, \hat{q}) = 0, & \hat{q} \in \hat{V}. 
				\end{aligned}  
		\] 
		Let $H:=\underline{H} \circ \Psi_g^{-1},$ then $H(q,0)=0$ for all  $q \in (-\delta, \delta) \times \hat{V}.$ 
  
		\textit{Step 3.} Consider the vector filed $\underline{X}(q):=\partial_p \underline{H}(q, 0)$ locally defined around $q=0$. Note that $\underline{X}(0)\ne 0,$ so by the flow-box theorem there exists diffeomorphism $\varphi_2,$ and an open neighborhood $V$ of $q=0$ such that the pushforward of $\underline{X}$ by $\varphi_2$ i.e. $(\varphi_2^{-1})_*(\underline{X})$ is the unit vector field $e_1$ on $V$. We define $ H:=\underline{H}\circ \Psi_{\varphi_2}^{-1}.$ Then, we have $\partial_{p} H(q, 0)=e_1$ for $q \in V.$  
  
		\textit{Step 4.} Let $\underline{H}$ be the Hamiltonian obtained in the previous step. Consider a linear coordinate change
		\[
				\varphi_3(q) = M q, 
		\]
		where $M \in GL(d+1, \R)$ has the following block form 
		$$
		M =  \bmat{1 & 0 \\ 0 & \bar{M}}.   
		$$
		Let $H:\underline{H} \circ \Psi^{-1}_{\varphi_3},$ so we have 
		\[
				H(q, p) = \underline{H}\big(\varphi_3(q), (D\varphi_3^{-1})^T p\big) 
				= \underline{H}(M^{-1} q, M^T p). 
		\]
		Note that (b) and (c) still hold for $H$ as we still have $H(q, 0) = k,$ and 
		\[
				\partial_p H(q, 0) = M \partial_p \underline{H}(q, 0) = M e_1 = e_1.
		\]
		Moreover, 
		\[
				\partial^2_{pp} H(0, 0) 
				= M \partial^2_{pp} \underline{H}(0, 0) M^T
				= \bmat{1 & 0 \\ 0 & \bar{M}} \bmat{* & * \\ * & \partial^2_{\hat{p} \hat{p}} \underline{H}(0, 0)} \bmat{1 & 0 \\ 0 & \bar{M}^T}.
		\]
		Since $\underline{A}(0) := \partial^2_{\hat{p} \hat{p}} H(0, 0)$ is a co-rank $m$ positively semi-definite matrix, there exists an orthogonal matrix $G$ such that
		\[
				\underline{A}(0)  = G^T \bmat{\Lambda & 0 \\ 0 & 0} G, 
		\]
		where $\Lambda$ is a $(d-m) \times (d-m)$ diagonal matrix with positive entries. We choose
		\[
				\bar{M} = \bmat{\Lambda^{-\frac12} & 0 \\ 0 & 0} G,  
		\]
		to have 
		\[
				\partial^2_{pp} H(0, 0)   = 
				\begin{bmatrix}
						* & * \\ * & 
						\begin{matrix}
								I_{d-m} & 0 \\ 0 & 0	
						\end{matrix}
				\end{bmatrix}.
		\]
\end{proof}

\subsection{Linearized transition maps } \label{sec:lin-trans}    

Throughout this section we assume that $\cD$ is a rank-$m$ distribution, $H$ is a $\cD$-Hamiltonian, and $\theta$ is an orbit in the $k$-energy level of $H$. We will assume that a normal form transformation has been performed so that the claim of \Cref{normalform} holds for $H$ and $\theta$. Consider the family of Hamiltonians 
$$
H_u=H+u, \quad u \in \mU(\theta),
$$
where $\mU(\theta)$ is defined in \eqref{eq:cU}. 

We use the notation
$x=(q,p),$ $q=(q_1,\hat{q}) \in \mathbb{R} \times \mathbb{R}^{d},$ and $p=(p_1,\hat{p})\in \mathbb{R} \times \mathbb{R}^d.$ Let $\Sigma_t=\{q_1=t\}$ and consider the one-parameter family of transition maps 
$$
R^t_u = P_{\Sigma_0, \Sigma_t}(\theta, H + u): \quad H_u^{-1}(k) \cap \Sigma_0 \to H^{-1}_u(k) \cap \Sigma_t
$$
subjected to the obit segment $(te_1,0)$ and the Hamiltonian vector field of $H_u.$ 

Since $dH_u(te_1,0)=(0,e_1)$, hence $H_u$ and $p_1$ have the same gradient along $(te_1, 0)$, i.e. $H^{-1}_u(k)$ is tangent to $\{p_1=0\}$ along the orbit segment. Note that
$$
dR^t_u:\{q_1=0,p_1=0\} \to \{q_1=0,p_1=0\}
$$
is a linear symplectic map.
\begin{lemma}
		Let $R^t_u$ be the one-parameter family of transition maps introduced in above. Then, $dR^t_u(0)$ coincides with the mapping $$\{q_1=0,p_1=0\} \ni \hat{x} \mapsto L(t)\hat{x}$$ where $L(t) \in Sp(2d,\mathbb{R})$ solves the following differential equation 
		\begin{equation}\label{eq_L}
				\dot{L}_u(t)=Y_u(t)L_u(t), \quad L_u(0)=I_d,
		\end{equation}
		in which $Y_u:=\partial^2_{\hat{x} \hat{x}}H_u(te_1,0).$    
\end{lemma}
\begin{proof}
Denote by $X$ the Hamiltonian vector field of $H_u.$ Let 
$$f(q,p)=1/\partial_{p_1}H_u(q,p),$$  and consider $\tilde{X}=f X$ which is a reparametrization of the vector field $X$. Let $\phi^t_u$ be the flow of $X$, and $\tilde{\phi}^t_u$ be the flow of $\tilde{X}$, then the vector fields $X$ and $\tilde{X}$ give rise to the same linearized transition map. For the flow of $\tilde{X}$, $\dot{q}_1$ is constant, hence the transition time of $\tilde{X}$ between the sections is constantly equal to $\delta$. Therefore
\[
L(t) := dR^t_u(0) = \Pi \, d\tilde{\phi}_u^\delta |_{\{q_1 = 0, p_1 = 0\}}
= \Pi \, d\tilde{\phi}_u^\delta \Pi^T
\]
where $\Pi$ is the $2d \times (2d + 2)$ projection matrix from $x$ to $\hat{x}$ components. Moreover, $\{q_1 = p_1 = 0\}$ is preserved by the flow $d\tilde{\phi}_u^t$, which means
\[
d\tilde{\phi}_u^\delta \Pi^T
= \Pi^T \Pi d\tilde{\phi}_u^\delta \Pi^T,
\]
we get
\[
\dot{L}(t) = 
\Pi \, d\tilde{X} d\tilde{\phi}_u^\delta \Pi^T
= 
\Pi \, d\tilde{X} \Pi^T \Pi d\tilde{\phi}_u^\delta \Pi^T
= 
\Pi \, d\tilde{X} \, \Pi^T \, L(t).
\]

Since
\begin{align*}
		d \tilde{X}(te_1,0)&=f(te_1)dX(te_1,0)+X(te_1,0)df(te_1,0)\\ &=dX(te_1,0)+[e_1 \hspace{1.2 mm} 0]^T df(te_1,0),
\end{align*}
we have
\[
\Pi \, d\tilde{X} \, \Pi^T
= \Pi \, dX \, \Pi^T
= 
\begin{bmatrix}
    0 & A \\ - \partial^2_{\hat{q}\hat{q}} u & 0
\end{bmatrix} = Y_u(t). 
\]
This concludes the proof.
\end{proof}

Let $e_i$ denote the Euclidean basis of $\R^d$, and
\[
		E_{ij} = e_i \otimes e_j + e_j \otimes e_i, 
\]
for $1 \le i \le j \le d$. Then $E_{ij}$ is a basis of the space $\mS(d)$ of $d \times d$ symmetric matrices. Set $w_{ij}(t) = \frac12 \partial^2_{\hat{q}_i \hat{q}_j} u(te_1)$, the equation (\ref{eq_L}) can be rewritten as
\begin{equation}\label{eq_2.9} 
		\dot{X}_w(t)=
		\begin{bmatrix}
				0  & A(t) \\ 
				0  & 0 
		\end{bmatrix}X_w(t)+ 
		\sum_{1 \le i \le j \le d } 
		w_{ij}(t)  
		\begin{bmatrix}
				0 & 0 \\ 
				E_{ij} & 0
		\end{bmatrix} X_w(t), \quad X_w(0)=I.
\end{equation}
We view \eqref{eq_2.9} as a control problem with $w_{ij}$ as the controls. 

For each $w \in L^2([0, \delta], \R^{d(d+1)/2})$, the control problem \eqref{eq_2.9} has a unique solution $X_w(t): \mI_w \subseteq [0,\delta] \to Sp(2d,\mathbb{R})$. Define $\mC \subset L^2([0, \delta], \R^{d(d+1)/2})$ as the set of controls $w$ for which $\mI_w=[0,\delta]$. This set is open and it contains $0$. 

\begin{definition}\label{def:mane-control}
		We say the curve $A= A(t): [0, \delta] \to \mS(d)$ is \emph{Man\'e controllable} if the end-point mapping $\mE: \mC \to Sp(2d,\mathbb{R})$ defined as  
  $$(w_{ij}) \mapsto X_w(\delta)$$ 
            is a submersion at $w = 0$.
\end{definition}

Since submersions are locally open mappings, we reduce the claim of Theorem \ref{pert_thm} to the controllability of $A(t)$:

\begin{proposition}\label{prop:reduce-to-mane}
		Suppose the curve $A(t)$ is \mane controllable, then for any neighbourhood $N$ of $\{\theta(t) \mid t \in (0, \delta)\}$, the mapping
		\[
				\cU(\theta) \cap \{\supp u \subset N\} \mapsto dR_u^\delta(0)
		\]
		is a submersion at $u = 0$.	
\end{proposition}
\begin{proof}
By assumption, the mapping
\[
		(w_{ij}) \mapsto X_w(\delta)
\]
is a submersion at $\omega = 0$, where $(w_{ij}) \in L^2([0, \delta], \R^{d(d+1)/2})$. Since $C^\infty_c$ is dense in $L^2$,  there exists a finite family
\[
	 w^{(k)} = (w_{ij}^{(k)}) \in C^\infty_c\big((0, \delta), \R^{d(d+1)/2}\big), \quad 1 \le k \le n, 
\]
such that the vectors
\[
		\partial_w \big|_{w = 0} X_w(\delta) \cdot w^{(k)}  
\]
span $Sp(2d, \R)$. Let $u^{(k)}$ be smooth functions satisfying
\[
		u^{(k)}(te_1, \hat{q}) = \sum_{i,j} w^{(k)}_{ij}(t) \hat{q}_i \hat{q}_j, \quad t \in [0, \delta]
\]
and $\supp u \subset N$, we have
\[
		\partial_u \big|_{u = 0} \left( dR_u^\delta(0)\right)  \cdot w^{(k)}
\]
span $Sp(2d, \R)$, in other words, $dR_u^\delta(0)$ is a submersion at $u = 0$.
\end{proof}

Assume that $H$ is already brought to the normal form of Proposition \ref{normalform}, defined in a local chart $V \subset \R^{d+1}$. We identify $T_q^*M$ with $\R^{d+1}$. Since in the normal form
\[
		e_1 = \partial_p H(q, 0) \in \cD_q 
\]
we have
\begin{equation}
		e_1.P=0 \quad  \text{for all} \quad P \in \mD^\perp_q.
\end{equation}
This implies that every vector in $\mD^\perp_q$ takes the form $(0, \hat{P})$. We then define $	\hat{\cD}^\perp_q$ to be the projection of $\cD^\perp_q$ to the $\hat{p}$ component. The following lemma follows immediately.

\begin{lemma}\label{hessian_rank}
		In the normal form system $H$, the hessian $\partial^2_{pp} H$ and $\partial^2_{\hat{p}\hat{p}} H$ both have co-rank $d + 1 - m$. Moreover, 
		\[
		  \partial^2_{\hat{p}\hat{p}} H(q, 0) \hat{P} = 0, \quad \text{for all} \quad
			\hat{P} \in \hat{\cD}^\perp_q.
		\]
		In particular,
		\[
				A(t) \hat{P} = 0 \quad \text{for all} \quad \hat{P} \in   \hat{\cD}^\perp_{te_1}.
		\]
\end{lemma}

\subsection{Mane controllability for co-rank 1 distributions }
\label{sec:mane-controllable}

From now on, we specialize to the case of co-rank $1$ distributions. Returning to the Hamiltonian system $H$ in the normal form and the orbit $\theta(t) = (te_1, 0)$. We fix the following notations:
\begin{itemize}
		\item $\tilde{\bn}(t) \in \R^{d+1}$ is the unit vector (in normal coordinates) in the direction of $\cD^\perp_{te_1}$.   
		\item $\bn(t) \in \R^d$ is the unit vector (in normal coordinates) in the direction of $\hat\cD^\perp_{te_1}$.   
\end{itemize}

\begin{lemma}\label{lem:n0}
		A time $t \in (0, \delta)$ is a $\cD$-regular time of $\theta(t)$	if and only if $\dot{\bn}(t) \ne 0$. 
\end{lemma}
\begin{proof}
	 Suppose $t$ is not a regular time.	Let $N(q)$ be a unit vector in the direction of $\cD^\perp_q$. By definition,
\begin{equation}  \label{eq:singular-time}
	   DN(te_1) e_1 \cdot v = 0   \quad \text{ for all $v \in \cD_{te_1}$.}
\end{equation}
	 On the other hand,
		\[
				0 = \frac{d}{dt}\|\tilde{\bn}(t)\|^2 = 	\frac{d}{dt} \|N(te_1)\|^2 = DN(te_1) e_1 \cdot N(te_1).
		\]
		Since $\cD_{te_1}$ and $N(te_1)$ span the whole tangent space, 
		\[
				\frac{d}{dt} \tilde{\bn}(t) = DN(te_1) e_1 = 0.   
		\]
		Conversely, $\frac{d}{dt} \tilde{\bn}(t) = 0$ implies \eqref{eq:singular-time} and $t$ is not a regular time.

		We established the equivalence between $t$ being a regular time and $\frac{d}{dt} \tilde{\bn}(t) \ne 0$. It remains to note that since $e_1 \in \cD_q$ in the normal coordinates               $\tilde{\bn}(t)$ has the following block form
		\[
				\tilde{\bn}(t) = \bmat{0 \\ \bn(t)}.
		\]
		Hence, $\frac{d}{dt} \tilde{\bn}(t) \ne 0$ if and only if $\dot{\bn}(t) \ne 0$.
\end{proof}

By Lemmas \ref{eq:singular-time} and \ref{hessian_rank}, if $0$ is a $\cD$-regular time then we have
		\begin{equation}  \label{eq:n0}
				\bn(0) = e_d, \quad	\dot{\bn}(0) \ne 0.   
		\end{equation}

\begin{definition}\label{def:cA}
	Let $\bn(t)$ be a $C^\infty$ smooth curve of unit vectors in $\R^d$ satisfying \eqref{eq:n0}. Define
		\[
\begin{aligned}
		  \cA_\delta(\bn)  
  = \Bigl\{ 
		& A = A(t) \in C^\infty([0, \delta], \mS(d)) \mid  A(0) = \bmat{I_{d-1} & 0 \\ 0 & 0}, \\
		& rank A(t) = d-1, \quad A(t) \bn(t) = 0 \text{ for all } t \in [0, \delta]
\Bigr\} . 
\end{aligned}
		\]
\end{definition}

We now state our main technical result on Man\'e controllability. The proof will be deferred to the end of the paper. 

\begin{theorem}\label{thm:controllable}
		Fix a curve $\bn(t)$ satisfying \eqref{eq:n0}. Then there exists an open and dense subset $\cV$ of $\cA_\delta(\bn)$, such that all $A \in \cV$ are \mane controllable (see Definition \ref{def:mane-control}). 
\end{theorem}

\subsection{Realization of perturbation and proof of Theorem \ref{pert_thm}}
\label{sec:proof-pert-thm}

In this section we finish the proof of Theorem \ref{pert_thm} assuming Theorem \ref{thm:controllable}. While Theorem \ref{thm:controllable} suggests that one can perturb the curve $A(t)$ to obtain \mane controllability, it is not clear whether this perturbation can be realized by a perturbation to $K$. This is covered by Proposition \ref{prop:realization} below.

\begin{proposition}[Realization of perturbation] \label{prop:realization}
		Suppose
  \begin{equation}\label{eq:H-dg}
		  H(q, p) = K(q, p + dg(q)) + U(q)  
  \end{equation}
		defined on $T^*V$ is in the normal form described by Proposition \ref{normalform}, where $K \in \cK_\cD^\star$, $\star \in \{\qd, \rf\}$ and $U \in C^\infty_b(V)$.

		 For a given $\tilde{A} \in \mA_\delta\big(\bn(t)\big)$, there exists $\tilde{K} \in \mK^\star(\mD, K, \theta_g)$ such that the Hamiltonian
		\[
				\tilde{H}(q, p) = \tilde{K}(q, p + dg(q))  + U(q)
		\]
		still satisfies the conclusion of Proposition \ref{normalform}, and
		\[
				\partial^2_{\hat{p} \hat{p}} \tilde{H}(te_1, 0) = \tilde{A}(te_1).
		\]
		Moreover, $\tilde{K}$ converges to $K$ as $\tilde{A}$ converges to $A(t) = \partial^2_{\hat{p}\hat{p}} H(te_1, 0)$.
\end{proposition}

Before proving Proposition \ref{prop:realization}, let us first prove a few properties of a generalized classical system viewed in the normal form given by \Cref{normalform}.
\begin{lemma}\label{lemma_2.5}
		Suppose $H=K\big(q,p+dg(q)\big)+U(q)$ with $K \in \cK_\cD^\beta$ for some $\beta > 1$ ($K$ is positively homogeneous of degree $\beta$). Assume that $H$ satisfies the conclusions of \Cref{normalform} on the neighborhood $V$. Let $B(q) = \partial^2_{pp} K(q, dg(q))$, then
		\[
				B(q) dg(q) = (\beta - 1) e_1, \quad
				(\partial_q B)(q) dg(q) = 0, \quad
				\forall q \in V.
		\]
\end{lemma}
\begin{proof}
		Since $K$ is $\beta$-homogeneous in $p$, Euler's theorem for homogeneous functions says
		\[
				\beta K(q, p) = \partial_p K(q, p) p.
		\]
		Differentiating both sides by $p$, we get
		\begin{equation}\label{eq:Hp-homo}
				(\beta - 1) \partial_p K(q, p) = \partial^2_{pp} K(q, p) p.   
		\end{equation}
		In particular,
		\[
				B(q) dg(q) = 	\partial^2_{pp} K(q, dg) dg = (\beta - 1) \partial_p K(q, dg) = (\beta - 1) \partial_p H(q, 0) = (\beta - 1) e_1.
		\]

		For the second identity, differentiate \eqref{eq:Hp-homo} by $q$ to get
		\[
				\partial_q (\partial^2_{pp}K)(q, dg) =  (\beta - 1)\partial_{qp} K(q, dg)
				= (\beta - 1) \partial^2_{qp}H(q, 0) = 0. 
		\]
\end{proof}
\begin{proof}[Proof of Proposition \ref{prop:realization}]
		Recall that the Hamiltonian takes the form \eqref{eq:H-dg}. Let 
		\[
				B(q) = \partial^2_{pp} H(q, 0) = \partial^2_{pp} K(q, dg(q)) = \bmat{* & * \\ * & C(q)}, 
		\]
		where $C(q) \in \mS(d)$. For $C_1(q) \in \mS(d)$ to be specified later, define
		\[
				B_1(q) := B(q) \bmat{0 & 0 \\ 0 & C_1(q)} B(q) 
				= \bmat{* & * \\ * & C(q) C_1(q) C(q)},
		\]
		and
		\[
				K_1(q, p)  := \frac12 B_1(q) p \cdot p, 
				\quad \tilde{K}(q, p) := K(q, p) + K_1(q, p). 
		\]
		Since $K_1 \in \cK_\cD^{\qd} \subset \cK_\cD^{\rf}$, this perturbation works in both $\cK_\cD^{\qd}$ and $\cK_\cD^{\rf}$.

		The following properties hold for $B_1$ and $K_1$ defined in above:
		\begin{itemize}
				\item From definition of $\mD^\perp,$ we have 
						\[
								B(q) \hat{P} = 0, \quad B_1(q) \hat{P} = 0, \quad \forall   \hat{P} \in \mD_q^\perp.
						\]
				\item $\partial_p K_1(q, dg(q)) = 0$. Indeed, using Lemma \ref{lemma_2.5} we have
						\[
\begin{aligned}
		\partial_p K_1(q, dg(q)) & = 	B_1(q) dg(q) =   B(q) \bmat{0 & 0 \\ 0 & C_1(q)} B(q) dg(q)  \\
														 & = B(q) \bmat{0 & 0 \\ 0 & C_1(q)} e_1 = 0. 
\end{aligned}
						\]
				\item Using Euler's theorem for homogeneous functions, we obtain
						\[
								K_1(q, dg(q)) = \frac12 \partial_p K_1(q, dg(q)) dg(q) = 0.   
						\]
				\item $\partial^2_{qp} K_1(q, dg(q)) = 0$. We use Lemma \ref{lemma_2.5} to obtain
						\[
								(\partial_q B_1)(q) dg(q)
								= \partial_q \left( B \bmat{0 & 0 \\ 0 & C_1}\right) B(q) dg(q)
								+  B \bmat{0 & 0 \\ 0 & C_1} (\partial_q B)(q) dg(q) = 0. 
						\]
		\end{itemize}
		Because $\partial_q K_1(q, dg(q)) = \partial_p K_1(q, dg(q)) = 0$, we know that 
		$$\tilde{K} = K + K_1 \in \mK^q_\mD(K, \theta_g).$$
		Moreover, since we have 
		\begin{align*}
				\tilde{H}(q, 0) &= H(q, 0) + K(q, dg(q)) = 0,  \\ 
				\partial_p \tilde{H}(q, 0) &= \partial_p H(q, 0) + \partial_p K(q, dg(q)) = e_1,
		\end{align*}
		the new system $\tilde{H}$ is still in normal form.

		It suffices to choose $C_1(q)$ so that
		\[
				\partial^2_{\hat{p} \hat{p}} K_1(q, dg(q))  
				= C(q) C_1(q) C(q)
		\]
		coincides with $R(t):=\tilde{A}(t)-C(te_1)$ at $q=te_1$. Let 
            $$G(q) = \bmat{f_1(q) & \cdots & f_d(q)}$$ 
            be a local orthonormal basis such that $\hat{\mD}_q^\perp = span\{f_{d-m+1}, \cdots, f_d\}$. As $C(q)$ has co-rank $m$, there exists a mapping $q \mapsto \bar{C}(q) \in GL(d-m)$ such that
		\[
				G^T(q) C(q) C_1(q) C(q) G(q)
				= \bmat{\bar{C}(q) & 0 \\ 0 & 0} G^T(q) C_1(q) G(q) \bmat{\bar{C}(q) & 0 \\ 0 & 0},
		\]
		and there exists $\bar{R}(t) \in  \mS (d-m)$ such that 
		\[
				G^T(te_1) R(t) G(te_1) =   \bmat{\bar{R}(t) & 0 \\ 0 & 0}. 
		\]
		It suffices to choose
		\[
				C_1(te_1)  = G(te_1) \bmat{\bar{C}^{-1}(te_1) \bar{R}(t) \bar{C}^{-1}(te_1) & 0 \\ 0& 0} G^T(te_1). 
		\]
\end{proof}

Combined with Proposition \ref{prop:realization}, we now prove the perturbation theorem Theorem \ref{pert_thm}.
\begin{proof}[Proof of Theorem \ref{pert_thm} assuming Theorem \ref{thm:controllable}]
		Let us denote by
		\[
				\underline{H} = \underline{K}(q, p) + \underline{U} 
		\]
		Consider the normal form system from Proposition \ref{normalform} adapted to the orbit $\underline\theta(t)$, given by
		\[
				H = \underline{H} \circ \Phi  = \underline{H} \circ \Psi_\varphi^{-1} \circ \Psi^g = K(q, p + dg) + U,
		\]
		where the transformed orbit is $\Phi \circ \underline\theta(t) = \theta(t) = (te_1, 0)$, $t \in [0, \delta]$ and set $A(t) = \partial^2_{\hat{p}\hat{p}}H(\theta(t))$. Let us denote the local section $\{q_1 = 0\}$ and $\{q_1 = \delta\}$ by $\Sigma_0$ and $\Sigma_\delta$.

		By Theorem \ref{thm:controllable}, there exists an open and dense set  $\cV \subset \cA_{\delta}(\bn)$ such that any $A_1 \in \cV$ is \mane controllable. By Prposition \ref{prop:realization}, any such $A_1$ sufficiently close to $A$ can be realized by a kinetic energy $T \in \cK^\star(\cD, K, \theta)$ sufficiently close to $K$. By Proposition \ref{prop:reduce-to-mane}, for any neighborhood $N$ of $\{\pi \circ \theta(t) \mid t \in (0, \delta)\}$,
		\[
				\cU(\theta) \cap \{\supp u \in N\} \mapsto dP_{\Sigma_0, \Sigma_\delta}(T + U + u)
		\]
		is a submersion at $u = 0$. 

		Let $\underline{T}, \underline{u}$ denote the corresponding functions in the orginial coordinates, denote $\underline\Sigma_0$, $\underline\Sigma_\delta$ the corresponding sections, and $\underline{N}$ be a neighorhood of $\underline\theta((0, \delta))$. Then
		\[
				\cU(\underline\theta) \cap \{\supp \underline{u} \in \underline{N}\} \mapsto dP_{\underline\Sigma_0, \underline\Sigma_\delta}(\underline{T} + \underline{U} + \underline{u})
		\]
		is a submersion at $\underline{u} = 0$.

		This proves that the conclusion of our proposition holds for a dense set of $T \in \cK^\star(\cD, K, \theta)$. The fact that this property is open is clear by continuity.
\end{proof}

\section{Parametric transversality and Proof of the main theorems}    
\label{sec:param-trans}

\subsection{Parametric transverality}

Let $H=K+T$ be a given generalized classical system where $K \in \mK_\mD^\star$, $\star \in \{\qd, \rf\}$. Throughout this section let $\Upsilon$ is a given nowhere dense $F_\sigma$ subset of $Sp(2d, \R)$. While conclusions of this section hold for any such $\Upsilon$, for the proof of the main theorem, $\Upsilon$ will be chosen to be the matrices in $Sp(2d, \R)$ that are admitting a root of unity as an eigenvalue. We use the notation $\phi^t(H,x)$ for the flow of Hamiltonian vector field of $H$. 

			 \begin{definition}\label{def:neat}
Let $\theta(t)=\big(Q(t),P(t)\big)$ be a periodic orbit with prime period $T$ of a Hamiltonian vector field. A time $t_0$ is called a \textit{neat time} if $\dot{Q}(t) \ne 0$, and $Q(t_0)$ is not a self-intersection point of $Q(t)$ i.e. there does not exists a time $s \ne t_0$ modulo T such that $Q(s)=Q(t_0).$ If $t_0$ is a neat time for $\theta(t)$, then we call $\theta(t_0)$ a \textit{neat point}. 
\end{definition}

Define
$$
\mathbb{P}:= ]0,\infty[ \times T^*M
\times \mK_\mD^\star \times  C^\infty(M).
$$
Denote by $\Pi:\mathbb{P}\to \mK_\mD^\star \times C^\infty(M)$ the projection mapping 
$$(s, x, K, U) \mapsto (K, U).$$ 
We define  
\begin{align*}
		\mathcal{P}^1 :=\big\{(s,x,K, U) 
		& \in  \mathbb{P}  \mid \text{$(K, U, k)$ is admissible, and $\phi^s(K + U, x) = x$} \\
		& \text{there exists a $\cD$-regular time on the orbit $\phi^t(K + U, x)$} \big\}. 
\end{align*}
We then define the subsets $\mathcal{P}^5 \subset \mathcal{P}^4 \subset \mathcal{P}^3 \subset \mathcal{P}^2 \subset \mathcal{P}^1$ as following 
\begin{align*}
		\mathcal{P}^2 &:= \big\{(s,x,K, U) \in \mathcal{P}^1 \mid \text{$s$ is the minimimal period for the orbit of $x$ }\big\}, \\
		\mathcal{P}^3 &:=\big\{(s,x,K, U) \in \mathcal{P}^2 \mid \text{there exists a time $t$ on the orbit} \\ 
									& \qquad \text{ that is both neat and $\cD$-regualr }\big\}, \\ 
\mathcal{P}^4 &:= \{(s,x,K,U) \in \mathcal{P}^4 \mid \text{ the linearized Poincar\'e map of $\phi^t(K + U, x)$} \\
							& \qquad \text{does not admit $1$ as an eigenvalue}\big\},  \\
		\mathcal{P}^5(\Upsilon) & := \big\{(s,x,K, U)  \in \mathcal{P}^4 \mid \text{the linearized Poincar\'e map of $\phi^t(K + U,x)$} \\ & \qquad  \text{ does not belong to $\Upsilon$}\big\}.
\end{align*}


\begin{proposition}\label{prop:p2p3} We have 
\[
  \mP^2 = \mP^3, 
\]
where $\mP^2$ and $\mP^3$ are defined above. 
\end{proposition}
The proof is given in the next subsection, the main idea is that any $\cD-$regular time interval must contain a dense set of neat times.

Recall that a meager set is a countable union of nowhere dense sets. We will assume Proposition \ref{prop:p2p3} and prove the following.
\begin{proposition}\label{prop:p3p5}
	$\Pi\big(\mP^3 \setminus \mP^5 \big) \subset \mK_\mD^\star \times C^\infty(M)$ 
 is a meager $F_\sigma$ subset.
\end{proposition}

The proof of Proposition \ref{prop:p3p5} will be done in three phases via Propositions \ref{pr_3.2},  \ref{prop_p3-p4}, and \ref{pr_p3-p4_Fsigma}. 

 The next lemma provides a sufficient condition under which the projection of an $F_\sigma$ subset is an $F_\sigma$ subset. A proof can be found in Section 3.1.1 of \cite{Aslanithesis}.
\begin{lemma}\label{lemma_3.1}
    Suppose that $X$ and $Y$ are two topological spaces and $X$ is a countable union of compact sets. Let $\pi_{Y}:X \times Y \to Y$ be the projection mapping $(x,y) \mapsto y.$ Assume that $C \subset X \times Y$ is closed, then $\pi_Y(C)$ is an $F_\sigma$ subset of $Y$.
\end{lemma}
\begin{proposition}\label{pr_3.2}
		$\Pi\big(\mP^4 \setminus \mP^5(\Upsilon) \big)$ is a meager $F_\sigma$ subset of $C^\infty(M) \times \mK_{\mD}$
\end{proposition}
\begin{proof}
		We will write $\mP^5(\Upsilon)$ as $\mP^5$ for short. Since the linearized Poincar\'e map taken at different points of the same orbit are conjugate, and the definition of $\mP^4$ and $\mP^5(\Upsilon)$ depends only on conjugacy class, both sets contain full orbits.

		For any orbit in $\mP^4$, we choose $(s_0, x_0, K, U)$ on the orbit such that $t = 0$ is both neat and $\cD$-regular for the orbit of $x$.  Let $\mY$ be an open neighborhood of this point in $\mathbb{P}$. We define $\mP_{loc}^4:=\mP_{loc} \cap \mY$, and $\mP_{loc}^5:=\mP_{loc}^4 \cap \mP^5$. Let 
		$$
		\rho(s_0,x_0)=(s_0,x_0, K, U) \in \{s_0\} \times \{x_0\} \times \mK_D^\star \times C^\infty(M)
		$$ 
		be a given section in $\mP^4_{loc}$.

		Consider the mapping 
				$$F_{(s_0,x_0)}(K, U): \rho(s_0,x_0) \to Sp(2d,\mathbb{R})$$ 
		defined as  
				$$
						(s_0,x_0,K, U) \mapsto dP,
				$$
		where $P$ is the Poincar\'e map associated with $\theta(t):=\phi^t(K + U, x_0)$ and Hamiltonian vector field of $K + U$, defined on the section $\Sigma$. Let $\Omega_{(s_0,x_0)}$ be the image of $F_{(s_0,x_0)}$. We will show the map  $F_{(s_0, x_0)}$ is weakly open, i.e. image of an open set contains an open set. Choose $\delta > 0$ and an intermediate section $\bar{\Sigma}$ to the orbit $\theta$ at $t = \delta$, and write $P$ as composition of two transition maps
		\[
				P = P_{\Sigma, \bar{\Sigma}}^{K, U} \circ P_{\bar{\Sigma}, \Sigma}^{K, U}.
		\]
		Let
		\[
				\mW = \{u \in \cU(\theta) \mid \supp u \cap \{\theta(t) \mid t \notin (0, \delta)\} = \emptyset\},  
		\]
		where $\cU(\theta)$ is defined in \eqref{eq:cU}. Since $(s_0, x_0, K, U) \in \mP^4$, $t = 0$ is both neat and $\cD$-regular. We now apply Theorem \ref{pert_thm} to $\theta(t)$ to get that there exists an open and dense $\cV \subset \cK^\star(\cD, K, \theta)$ such that for all $T \in \cV$, the mapping
\begin{equation}  \label{eq:dp-open}
			 \mW \ni u \mapsto 	dP_{\Sigma, \bar{\Sigma}}^{K + T, U + u}
\end{equation}
is a submersion at $u = 0$. Furthermore, $P_{\bar{\Sigma}, \Sigma}^{K + T, U + u}$ is constant in $u$ for all $u \in \mW$. It follows that $F_{(s_0, x_0)}(T, \cdot)$ is a submersion near $u = 0$, hence an open mapping for $u$ close to $0$. 

Since $T$ can be chosen to be arbitrarily close to $K$, the mapping $F_{(s_0, x_0)}(\cdot, U + \cdot)$ maps any neigborhood of $(K, 0) \in \cK^\star(\cD, K, \theta) \times \mW$ to a set with non-empty interior, i.e. $F_{(s_0, x_0)}$ is weakly open. Note that preimage of nowhere dense set under a weakly open mapping is nowhere dense, and preimage of $F_\sigma$ set under a continuous mapping is $F_\sigma$, we obtain 
 $$F^{-1}_{(s_0,x_0)}\big(\Upsilon \cap \Omega_{(s_0,x_0)}\big) \subset \rho(s_0,x_0)$$ 
 is a meager $F_\sigma$.

By definition of $\mP^5$, we have  
\begin{equation}
		\bigcup_{(s_0,x_0)}F^{-1}_{(s_0,x_0)}\big(\Upsilon \cap \Omega_{(s_0,x_0)}\big)=\mP^4_{loc} \setminus \mP^5_{loc}. 
\end{equation}
For each $(s_0,x_0)$, $F^{-1}_{(s_0,x_0)}\big(\Upsilon \cap \Omega_{(s_0,x_0)}\big)$ is a meager $F_\sigma$ subset of $\mP^4_{loc}$. Now we aim to show that $$
\bigcup_{(s_0,x_0)}F^{-1}_{(s_0,x_0)}\big(\Upsilon \cap \Omega_{(s_0,x_0)}\big)$$ can be written as a countable union of meager $F_\sigma$ sets.

We write $(s_0,x_0) \overset{(K, U)}{\sim}(s_0,x_1)$ if and only if there exists $t_0 \in \mathbb{R}$ such that $\phi^{t_0}(K + U,x_0)=x_1$, and the prime period of the closed orbit $\phi^{t_0}(K + U,x_0)$ is $s_0$. This is an equivalence relation over $\mP^4_{loc}$. Note that if $(s_0,x_0) \overset{(K, U)}{\sim}(s_0,x_1)$ then the linearized Poincar\'e maps $dP_{(s_0,x_0)}(K, U)$ and $dP_{(s_0,x_1)}(K, U)$ are conjugate to each other. Hence, because $\Upsilon$ is invariant under conjugacy we can write 
\begin{equation}\label{eqq_3.2}
\bigcup_{\big(s_0,x_0\big)}F^{-1}_{(s_0,x_0)}\big(\Upsilon \cap \Omega_{(s_0,x_0)}\big)
				 = 
				\bigcup_{\big(s_0,[x_0]\big)}F^{-1}_{(s_0,x_0)}\big(\Upsilon \cap \Omega_{(s_0,x_0)}\big).
\end{equation}
Since $F$ is defined on $\mP^4$ and the set of closed non-degenerate orbits of order 1 is at most countable, the right hand side of (\ref{eqq_3.2}) is a countable union of meager $F_\sigma$ sets.

So far, we proved that $\mP^4_{loc} \setminus \mP^5_{loc}$ is a meager $F_\sigma$ subset of $\mP^4_{loc}$. By Lemma \ref{lemma_3.1}, we deduce that $\Pi \big(\mP^4_{loc} \setminus \mP^5_{loc}\big)$ is as $F_\sigma$ subset of $\Pi(\mY)$. 

Non-degeneracy of the orbit $\phi^t(K + U, x)$ depends continuously on the parameter $U$, so the restriction of $\Pi$ to $\mP^4_{loc}$ is a homeomorphism. Hence, $\Pi\big(\mP^4_{loc} \setminus \mP^5_{loc}\big) \subset \Pi(\mY)$ is meager. 

In conclusion, because $\mathbb{P}$ is separable, $\Pi(\mP^4 \setminus \mP^5)$ is a meager $F_\sigma$  subset of $\mathbb{P}.$ 
\end{proof}

    \begin{lemma}\label{lemma_open-inclusions}
        The inclusions $\mP^4 \subset \mP^3$ and $\mP^2 \subset \mP^1$ are both open. 
    \end{lemma}
    \begin{proof}
           By definition, $\mP^4 \subset \mP^3$ is open, so we only need to show that $\mP^2 \subset \mP^1$ is an  open inclusion. 

            Now we prove that $\mP^2 \subset \mP^1$ is open by showing that $\mP^1 \setminus \mP^2 \subset \mP^1$ is closed. Take a sequence $(s_n,x_n,K_n, U_n) \in \mP^1 \setminus \mP^2$. Let $S_n$ be the minimal period of $(s_n,x_n,K_n, U_n)$. There exists an integer sequence $i_n \geq 2$ such that $s_n=i_nS_n$. Let $(s,x,K, U)$, $S$, and $i$ be the limits of $(s_n,x_n,K_n, U_n)$, $S_n$, and $i_n$ respectively. There exists a strictly increasing function $\rho: \mathbb{N} \to \mathbb{N}$ such that $s_{\rho(n)}=i_{\rho(n)}S_{\rho(n)}$ and $i_{\rho(n)}$ is either constant or tends to $\infty$. If $i_{\rho(n)}$ is constant then $\frac{s}{i}=S$ which means $s$ is not the minimal period, so $(s,x,K, U) \notin \mP^2$. If $i_{\rho(n)} \to \infty$ then $x$ is a stationary point which can not occur in a supercritical energy level. 
    \end{proof}

        We aim to prove that $\Pi(\mP^3 \setminus \mP^4)$ is an $F_\sigma$ subset of $C^\infty(M)$. To do so, by Lemma \ref{lemma_3.1}, we only need to show that  $\mP^3\setminus \mP^{4} \subset \mathbb{P}$ is $F_\sigma$. 

        A subset is \textit{locally closed} if it is an intersection of a closed and an open subset. All locally closed subsets of a metrizable space are $F_\sigma$. The proof of the mentioned claim is elementary and can be found in page 43 of \cite{Aslanithesis}. 

    \begin{proposition} \label{pr_p3-p4_Fsigma} 
         $\mP^3\setminus \mP^{4}$ is an $F_\sigma$ subsets of $\mathbb{P}$.
    \end{proposition}
        \begin{proof}
						First, we need to note that $\mP^1 \subset \mathbb{P}$ is locally closed. That is because the set $\mD$ regular points is an open subset of the phase space, and the set of periodic points is closed. By Proposition \ref{prop:p2p3} and Lemma \ref{lemma_open-inclusions}, we have   $\mP^4 \subset \mP^3 = \mP^2 \subset \mP^1$ where all the inclusions are open. Then $\mP^1 \subset \mathbb{P}$ being locally closed implies $\mP^3 \setminus \mP^4$ is locally closed. So, $\mP^3 \setminus \mP^4$ is $F_\sigma$. 
        \end{proof}

\begin{proposition}\label{prop_p3-p4}
     $\Pi\big(\mP^3 \setminus \mP^4\big)$ is a nowhere dense $F_\sigma$ subset of $\mK_\mD^\star \times C^\infty(M)$.
\end{proposition}

   Let $(s_0,x_0,K, U) \in \mP^3$ be given such that $t = 0$ is both neat and $\cD$-regular, and Let $\theta(t)$ denote the orbit $\phi^t(K + U, x_0) = \phi^t(K + T + U, x_0)$.  Given
	 \[
	   T \in   \cK^\star\big(\cD, K, \theta\big)
	 \]
	 (see \eqref{eq:admissible-K}) sufficiently small, then $(s_0, x_0, K + T, U) \in \mP^3$ with $t = 0$ both neat and $\cD$-regular. Let $\Lambda = T^*\Gamma$ where $\Gamma \subset M$ is a local transversal section at $x_0$ to $\pi \circ \theta(t)$. 

	 We consider the space of potential perturbations
	 \[
			 \mW = \{u \in C^\infty(M) \mid \supp u \cap \Gamma = \emptyset \}.
	 \]
	 Then the set
	 \[
			 \Lambda_k = \Lambda \cap \{ K + T + U + u = k\}  
	 \]
	 is independent of $u \in \mW$ and is invariant under the Poincar\'e map $P_\Lambda(u)$. Let
	 \[
			 P_{\Lambda, k}(u, \hat{x}) = P_\Lambda(u)(\hat{x})
	 \]
	 for $\hat{x} \in \Lambda_k$ be the restricted map.

	 The following proposition imply Proposition \ref{prop_p3-p4}.
	 \begin{proposition}\label{prop:control-order1}
For any neighborhood $\mN$ of $0$ in $\cK^\star\big(\cD, K, \theta\big)$, there exists $T \in \mN$, such that the linearized map
\[
		d_u P_{\Lambda, k}(0, x_0)
\]
is a surjection.
\end{proposition}

\begin{proof}[Proof of Proposition \ref{prop_p3-p4} using Proposition \ref{prop:control-order1}]
		Let $T$ be given by Proposition \ref{prop:control-order1}.  Since $d_u P_{\Lambda, k}$ is a submersion at $u = 0$, near $u = 0$, the set
		\[
				N = \left\{ (\hat{x}, u) \in \Lambda_k \times \mW \mid 
				P_{\Lambda, k}(u, \hat{x}) - \hat{x} = 0 \right\}
		\]
		is a submanifold in the neighborhood of $(x_0, 0)$ in $\Lambda_k \times \mW$. 

		Let $\hat{\Pi}(s, x, K, U) = U$, then there exists $s > 0$ such that $(s, \hat{x}, K + T, U + u)$ is in $\mP^4$ if and only if $u$ is a regular value of the mapping $\hat{\Pi}|_N$. By Sard's Theorem, such values are dense in $\hat{\Pi}(N)$. In particular, there exists $u$ arbitrarily close to $0$ such that $(s, \hat{x}, K + T, U + u) \in \mP^4$. Since $T$ can be chosen to be arbitrarily close to $0$ as well, we obtain that $\mP^4$ is dense in $\mP^3$. Since we already showed $\mP^3 \setminus \mP^4$ is $F_\sigma$, we conclude.
\end{proof}

We now prove Proposition \ref{prop:control-order1}. Apply Proposition \ref{normalform} to $H = K + U$ and the orbit $\theta$, we assume that the Hamiltonian is in normal form on a neighborhood $V$ of $x_0$. We choose the section $\Lambda = T^*\Gamma$ so that $\Gamma = \{q_1 = 0\}$ in the normal form coordinate. Note that conclusion of Proposition \ref{prop:control-order1} is independent of the choice of sections. 
\begin{lemma}\label{lem:kinetic-pert}
Given $(s_0, x_0, K, U) \in \mP^3$,  let $\mN$ be a neighorhood of $0$ in $\cK(\cD, K, \theta)$, there exists $T \in \mN$ and $\delta > 0$ such that in the normal form coordinates	
\[
		\partial_t \partial^2_{\hat{p} \hat{p}} (H + T)(te_1, 0)
\]
is non-degenerate for all $t \in [0, \delta]$.
\end{lemma}

The following lemma is a corollary of Lemma \ref{lem:der-A}. 
\begin{lemma}\label{Adot-blockform-lemma}
		Assume that $H$ is in the normal form given by Proposition \ref{normalform}. Suppose that $\partial_pH(te_1,0)=:A(t) \in \cA_\delta(\bn),$ and $\dot{\bn}(t) \ne 0$ for $t \in [0,\delta],$ then $\dot{A}(0)$ has the following block form 
		\begin{equation}\label{Adot-blockform}
				\dot{A}(0)= \begin{bmatrix}
						\mu_1 & & &  \\
									& \ddots  & & v   \\
									& & \mu_{d-1}  &  \\ 
									& -v^T & & 0
				\end{bmatrix},
		\end{equation}
		where $v= \begin{bmatrix}
				v_1 & v_2 & \hdots & v_{d-1}
		\end{bmatrix}^T$ is not the zero vector.    
\end{lemma}

\begin{proof}[Proof of Lemma \ref{lem:kinetic-pert}]
		By Lemma \ref{Adot-blockform-lemma}, $\dot{A}(0):=\partial_t \partial^2_{\hat{p}\hat{p}}\tilde{T_0}(te_1,0)$ viewed in the coordinates given by Proposition \ref{normalform} has a block form similar to (\ref{Adot-blockform})
		We have 
		$$
		\det \dot{A}(0)=\sum_{k=1}^{d-1} \big(v_k^2 \prod_{i \ne k} \mu_i\big).
		$$
		$\big\{(\mu_1, \hdots, \mu_{d-1}) \in \mathbb{R}^{d-1} \mid \sum_{k=1}^{d-1} \big(v_k^2 \prod_{i \ne k} \mu_i\big)=0\big\}$ is an algebraic subset of $\mathbb{R}^{d-1}$. Therefore there exist arbitrarily small perturbation $A_1 \in \cA_\delta(\bn)$ such that
		\[
				\det(\dot{A}(0) + \dot{A}_1(0)) \ne 0.
		\]
		By Proposition \ref{prop:realization}, there exists  $T \in \mK^\star(\cD, K, \theta)$ such that
		\[
				\partial^2_{\hat{p}\hat{p}}(H + T)(te_1, 0) = A(t) + A_1(t)  
		\]
		for $t$ clsoe to $0$. Moreover, $T$ converges to $0$ as $A_1$ converges to $0$. This proves our lemma.  
\end{proof}

 Recall that the normal form is valid in a neighborhood $V$ of $x_0$. Let $0 < \delta_1 < \delta_2 < \delta$ to be specified later, consider a section
 \[
		 \bar{\Lambda} = T^*\Gamma_{\delta_2}, \quad
		 \text{where $\Gamma_{\delta_2}$ is the section $\{q_0 = \delta_2\} \subset M$ in normal form coordinates}.
 \]
 Let
 \[
		 P_{\Lambda, \bar{\Lambda}, k}, \quad P_{\bar\Lambda, \Lambda, k}  
 \]
 be the transition map between the given sections restricted to energy $k$, then
 \[
		 P_{\Lambda, k}(u, \hat{x}) =  P_{\bar\Lambda, \Lambda, k}  \circ P_{\Lambda, \bar{\Lambda}, k} (u, \hat{x}).
 \]
 We will further restrict our potential perturbation $u$ to the set 
 \[
		 \mW' = \{u \in C^\infty(M) \mid  \supp u \cap \{\pi \circ \theta(t) \mid t \notin (0, \delta_2)\} = \emptyset\} .
 \]
As such, the map $\partial_u P_{\Lambda, k}(0, x_0)$ is a surjection if and only if $\partial_u P_{\bar\Lambda, \Lambda, k}(0, x_0)$ is. 

\begin{proof}[Proof of Proposition \ref{prop:control-order1}]
		Let $T$ be chosen as in Lemma \ref{lem:kinetic-pert}.  

Let us define
\[
		\psi_u(\hat{q}, \hat{p})   
\]
to be the map $P_{\Lambda, \bar\Lambda, k}$ written in the coordinates $(\hat{q}, \hat{p})$, then
\[
		\psi_u = \pi_{(\hat{q}, \hat{p})} \phi_u^{\tau(\hat{q}, \hat{p}, u)}(0, \hat{q}, 0, \hat{p})  
\]
where $\phi_u^t$ is the flow of $H_u = H + T + u$, where $u \in \mW'$, $\pi_{(\hat{q}, \hat{p})}$ is the standard projection, and $\tau(\hat{q}, \hat{p}, u)$ is the transition time between the sections $\Lambda$ and $\bar{\Lambda}$. Note that $\tau(0, 0, 0) = \delta_2$. 

Then we have
\[
\begin{aligned}
		\partial_u \psi_u(0, 0)  \big|_{u = 0}
		&=   d\pi_{(\hat{q}, \hat{p})} 
 \left( \partial_u \phi_u^{\delta_2}(0) + X_{H_u}^{\delta_2} \circ \phi_u^t(0) \cdot \partial_u \tau(0, 0)  \right) \\
		& =  d\pi_{(\hat{q}, \hat{p})} \partial_u \phi_u^{\delta_2}(0) 
\end{aligned}
\]
since $ X_{H + T + u}^{\delta_2} \circ \phi_u^t(0) = e_1$ and has zero projection to $(\hat{q}, \hat{p})$.

Consider a smooth approximation $\eta_\epsilon(t)$ supported on $(0,\delta_2)$ of Dirac delta function $\mu(t - \delta_1)$. We take a family of functions $h_{ij} \in C^\infty(M)$, where $i,j \in \{1,2, \hdots d\},$ such that $h_{ij}(q)(te_1)=0$ and 

$$-\partial_{\hat{q}}h_{ij}(t e_1)=\eta_\epsilon(t) e_j + \ddot \eta_\epsilon(t)e_i.$$ 

Since 
\begin{equation}
	\partial_u\phi^t(x,u)(h)= \partial_x \phi^t(x,u) \int_{0}^{t}  [\partial_x \phi^s(x,u)]^{-1} 
	\begin{bmatrix}
			0 \\ 
			-d h \big(\pi \circ \phi^s(x,0)\big) 
	\end{bmatrix} d s, 
\end{equation}
we denote $V(\sigma)=\phi_u^\sigma(0)\big|_{u = 0}$, then
		 \begin{align}
		 \partial_u\psi(0,0) \big|_{u = 0}(h_{i j}) \,
		 &= d\pi_{(\hat{q}, \hat{p})} 
		 V(\delta_2) \int_{0}^{\delta_2}  V^{-1}(s) 
		\begin{bmatrix}
		0 \\ 
		\bmat{0 \\ \eta_\epsilon(s) e_j }
		\end{bmatrix} 
		d s \label{vardiff_reg} \\ 
		 & + d\pi_{(\hat{q}, \hat{p})} V(\delta_2) \int_{0}^{\delta_2}  V^{-1}(s) 
		\begin{bmatrix}
		0 \\ 
		\bmat{0 \\ \ddot \eta_\epsilon(t)e_i}
		\end{bmatrix} d s   \label{vardiff_der}
 \end{align}
 Let us compute \eqref{vardiff_der} separately. First note the following general calulation for Hamiltonian flow of $H$:
 \begin{equation}  \label{eq:V-inverse}
\begin{aligned}
		 \frac{d}{dt} (V^{-1}(t) DX_H(\phi^t))
		 & = \frac{d}{dt}(V^{-1}(t)) D X_H + V^{-1}(t) \frac{d}{dt} DX_H(\phi^t) \\
		 & = - V^{-1}(t) (D X_H)^2 + V^{-1}(t) \frac{d}{dt} DX_H(\phi^t).
\end{aligned}
 \end{equation}
 In our normal form, at $u = 0$, 
 \[
		 DX_{H_0}(\phi^t(0)) = DX_{H_0}(te_1)
		 = \begin{bmatrix}
				 0 & \bmat{* & * \\ * & A} \\
				0  & 0
		 \end{bmatrix},  
 \]
 hence
 \begin{equation}  \label{eq:DXH-2}
      \left( DX_{H_0}(\phi^t(0))\right)^2 = 0.
 \end{equation}

 We have
\begin{equation}  \label{eq:dirac-der}
 \begin{aligned}
		 & V(\delta_2) \int_{0}^{\delta_2}  V^{-1}(s) 
				\begin{bmatrix}
				0 \\ 
				\bmat{ 0 \\ \ddot\eta_\epsilon(t)e_i}
				\end{bmatrix} d s  \\
		&= -V(\delta_2) \int_{0}^{\delta_2} - V^{-1}(s) \mathbb{J}\partial^2_{(q, p)}H(s e_1,0) 
				\begin{bmatrix}
					0 \\ 
					\bmat{0 \\ \dot\eta_\epsilon(s)e_i}
				\end{bmatrix} ds \\
		&=V(\delta_2) \int_0^{\delta_2} V^{-1}(s) 
				\begin{bmatrix}
						0 & \bmat{* & * \\ * & A} \\ 
						0 & 0
				\end{bmatrix}
				\begin{bmatrix}
					0 \\ 
				\bmat{0 \\ \dot\eta_\epsilon(s)e_i}
				\end{bmatrix} ds \\
		&=- V(\delta_2) \int_0^{\delta_2} V^{-1}(s) 
				\begin{bmatrix}
						0 & \bmat{* & * \\ * & A'(s)} \\ 
						0 & 0
				\end{bmatrix}
				\begin{bmatrix}
					0 \\ 
				\bmat{0 \\ \eta_\epsilon(s)e_i}
				\end{bmatrix} ds \\
		 &\approx V(\delta_2)V^{-1}(\delta_1) \begin{bmatrix}
				 \bmat{* \\ A'(\delta_1)e_i} \\
			 0
		\end{bmatrix}
 \end{aligned}
\end{equation}
where the fourth line uses \eqref{eq:V-inverse} and \eqref{eq:DXH-2} where $\approx$ means equal in the limit $\epsilon \to 0$. We combine \eqref{vardiff_reg} and \eqref{eq:dirac-der} to get
\[
  \partial_u \psi(0, 0) \big|_{u = 0} \approx
	d\pi_{(\hat{q}, \hat{p})} V(\delta_2)V^{-1}(\delta_1)
 \bmat{* \\ A'(\delta_1)e_i\\ 	* \\ e_j  } .
\]
Since $\dot{A}(0) \ne 0$, $\dot{A}(\delta_1)$ is invertible for $\delta_1$ sufficiently small. By choosing $\delta_2 - \delta_1$ small enough, $V(\delta_2)V^{-1}(\delta_1)$ is close to $I$ and $\partial_u \psi(0, 0)|_{u = 0}$ is onto.
\end{proof}

\subsection{Tagged segments}

Recall the definition of a neat time in Definition \ref{def:neat}. An orbit of a Hamiltonian vector field is called \textit{neat} if it admits a neat point.

    A Hamiltonian $H$ is called reversible if $H=H \circ \Re$ where $\Re:T^*M \to T^*M$ is the \textit{reversing involution} defined as $\Re(\alpha_q)=-\alpha_q,$ for all $\alpha_q \in T^*_qM$. Note that if $\big(Q(t),P(t)\big)$ is an orbit of a reversible Hamiltonian $H$ so is $\big(Q(-t),-P(-t)\big)$. An orbit $\big(Q(t),P(t)\big)$ of a reversible Hamiltonian is called \textit{reversible} (or \textit{symmetric}) if it is the same orbit as its time reversal $\big(Q(-t),-P(-t)\big)$.

		The main reason that we require our Hamiltonian to be reversible and the energy to be supercritical is the following.
\begin{lemma}
In a reversible Hamiltonian $H(q, p)$, there are no reversible orbits in a supercritical energy level.	
\end{lemma}
\begin{proof}
	If $\big(Q(t),P(t)\big)$ is a reversible orbit, then $\big(Q(t),P(t)\big)$ and $\big(Q(-t),-P(-t)\big)$ are representing the same orbit. Each reversible closed orbit has two points of intersection with the zero section, and the velocity of the projected orbit is zero on these two points. Recall that if $\big(Q(t),P(t)\big)$ is an orbit in a supercritical energy level of generalized classical system then $\dot{Q}(t)$ never vanishes.
\end{proof}


    \begin{definition}
        An orbit segment $\theta|_{[t_0,t_1]}$ is called a \textit{tagged segment} if it does not admit any neat points. The orbit segment is called \emph{strongly neat} if it does not admit any tagged sub-segment.
    \end{definition}

		We prove in this section that:
    \begin{proposition}\label{prop-stronglyneat}
          Any $\mD$-regular orbit segment in a supercritical energy level of a $\mD$-Hamiltonian $H=K+U$, $K \in \mK_\mD^{\rf}$, is strongly neat. 
    \end{proposition}

		\begin{proof}[Proof of Proposition \ref{prop:p2p3} using Proposition \ref{prop-stronglyneat}]
				Let $(s_0, x_0, T, U) \in \mP^2$ and $\theta(t)$ be the orbit of $(s_0, x_0, T, U)$, then there exists a regular time interval $[t_0, t_1]$ for $\theta$. Proposition \ref{prop-stronglyneat} implies that $[t_0, t_1]$ is strongly neat, which implies that the neat times are dense in $[t_0, t_1]$. In particular, there exists $t$ that is both neat and $\cD$-regular, i.e.
				\[
				  \mP^2 \subset \mP^3.  
				\]
				Since $\mP^3 \subset \mP^2$, we conclude.
\end{proof}

We continue with the proof of Proposition \ref{prop-stronglyneat}.

    \begin{proposition}\label{tagged_interval}
				Assume that $\theta$ admits a tagged segment then there exists two disjoint interval $[t_0,t_1], [\tau_0,\tau_1] \in \mathbb{R}/s_0\mathbb{Z}$ such that either the projections of orbit segments coincide:
				\[
						\pi \circ \theta(t + t_0) = \pi \circ \theta(t + \tau_0), \quad \forall 0 \le t \le t_1 - t_0
				\]
				or they are time-reversals of each other:
				\[
				  		\pi \circ \theta(t + t_0) = \pi \circ \theta(- t + \tau_1), \quad \forall 0 \le t \le t_1 - t_0.  
				\]
    \end{proposition}
    To prove the above proposition first we need to state the following lemma.
\begin{lemma}\label{finite_curves}
Let $I_0, \cdots, I_k$ be non-empty closed intervals in $\R/T\Z$, and $Q_i: I_i \to \R^d$ be homeomorphisms onto their images. Suppose 
\[
		\bigcup_{i = 1}^k Q_i(I_i) \supset  Q_0(I_0), 
\]
then there exists $i \in \{1,2, \hdots, k\}$ and a nontrivial open interval $J \subset I_i$ such that
\[
  Q_i(J) \subset Q_0(I_0).
\]
\end{lemma}
\begin{proof}
		For $t \in I_0$, define
\[
		\mI(t)  = \{ i \in \{1, \dots, k\} \mid  \exists s \in I_i \text{ such that } Q_i(s) = Q_0(t)\} .
\]
Suppose $t_k \in I_0$ is such that $t_k \to t$ and $i \in \mI(t_k)$, then there exists a subsequence $k_j$ and $s_j \in I_i$ such that $Q_i(s_j) = Q_0(t_{k_j})$ and $s_j \to s \in I_i$. It follows that $Q_i(s) = Q_0(t)$, so $i \in \mI(t)$. This means the set
\[
		\mT_i = \{t \in I_0 \mid i \in \mI(t) \}  
\]
is a closed set. Since  by assumption $\bigcup_{i = 1}^k \mT_i = I_0$, by Baire's category theorem, at least one of $\mT_i$ must have non-empty interior. Let $J' \subset \mT_i$ be an nonempty open interval, then $i \in \mI(t)$ for all $t \in J'$, i.e.
\[
		Q_i(I_i) \supset Q_0(J).  
\]
Since $Q_i$ is a homeomorphism, the set $Q_i^{-1}(Q_0(J))$ is a topological open segment, which equals $Q_i(J)$ for some $J \subset I_i$.
\end{proof}
    \begin{proof}[Proof of \Cref{tagged_interval}]
        Let $ \tilde{q} \in \{\pi \circ  \theta (t) \mid t \in [t_0,t_1]\}$ be given. We claim that the following set is finite 
            $$
               \Omega(\tilde{q})=\{t \in [0,T] \mid \pi \circ \theta(t)=q_0\}.
            $$
        We argue by contradiction: If $\Omega(q_0)$ is not finite then there exists a sequence $t_k$ such that $\pi \circ \theta(t_k)=\tilde{q}.$ Up to taking a subsequence, $t_k$ has a an accumulation point $t_0$, and $\pi \circ \dot{\theta}$ vanishes at $t_0$, which contradicts with the supercriticality of the energy. 

        Let's say $\Omega(q_0)$ has $k+1$ elements, namely $t_0,t_1,t_2, \hdots,t_k$. Consider the projected orbit segments $Q_i(t):=\pi \circ \theta([t_i-\delta,t_i+\delta])$ where $\delta>0$ is a chosen sufficiently small such that $I_i:=[t_i-\delta,t_i+\delta]$ are disjoint intervals and $Q_i$ are homeomorphisms onto their images. By Lemma \ref{finite_curves}, there exists $j$ and $l$ in $\{0,1,2, \hdots, k\}$, and $\delta_1,\delta_0>0$ such that
				\[
				  \pi \circ \theta\big((t_j-\delta_0,t_j+\delta_0)\big)=\pi \circ \theta\big((t_l-\delta_1,t_l+\delta_1)\big).
				\]
				Let $K_0 = (t_j - \delta_j, t_j + \delta_j)$ and $K_1 = (t_l - \delta_l, t_l + \delta_l)$. 
				Since $\pi \circ \theta$ is a local diffeomorphism, we may assume that $\delta_j, \delta_l$ are small enough that 
				\[
						\sigma := \left( (\pi \circ \theta)|_{K_1}\right)^{-1}  	\circ	(\pi \circ \theta)|_{K_0}  
				\]
				is a diffeomorphism.

        Since $\pi \circ \theta(t)=\pi \circ \theta\big(\sigma(t)\big)$ for all $t \in K_0$ we have
				\[
            d\pi \circ \dot{\theta}(t) = d\pi \circ \dot{\theta}\big(\sigma(t)\big)\dot{\sigma}(t),
				\]
				and hence if we let $\theta(t)=\big(Q(t),P(t)\big)$,
				\[
            \partial_p H\big(Q(t),P(t)\big) =\dot{\sigma}(t) \partial_pH\big(Q(\sigma(t)),P(\sigma(t))\big)=\partial_pH\big(Q(\sigma(t)),\dot{\sigma}(t)P(\sigma(t))\big),
				\]
				where the last equality hols because $\partial_pH(q,p)$ is 1-homogeneous. Since $H(Q(t), \cdot)$ is strictly convex on $T^*_{Q}M/\cD^\perp$, $\partial_p H(Q, \cdot)$ is one-to-one on  $T^*_{Q}M/\cD^\perp$. Therefore
\begin{equation}  \label{eq:Psigma}
						P(t) - \dot{\sigma} P(\sigma(t)) =: \eta(t) \in \cD^\perp(Q(t)).
\end{equation}

				By Euler's theorem for homogeneous functions, and because $\big(Q(t),P(t)\big)$ and $\big(Q(\sigma(t)),P(\sigma(t))\big)$ have the same energy, we have 
        \begin{align*}
\big\langle \partial_p H\big(Q(t),P(\sigma(t))\big), P\big(\sigma(t)\big)\big\rangle 
& = \big\langle \partial_p H\big(Q(t),P(t)\big), P(t)\big\rangle \\
& = \dot{\sigma}(t) \big\langle \partial_p H\big(Q(t),P(\sigma(t))\big), P\big(t\big)\big\rangle.
        \end{align*}
				Therefore
            \begin{align*}
								0 & =  \big\langle \partial_p H\big(Q(t),P(t)\big), P(\sigma(t))-\dot{\sigma} P(t)\big\rangle \\
									& =  \big\langle \partial_p H\big(Q(t),P(t)\big), P(\sigma(t))-\dot{\sigma}^2
										 P(\sigma(t)) - \dot\sigma \eta(t)
									\big\rangle \\
									& =  \big\langle \partial_p H\big(Q(t),P(t)\big), P(\sigma(t))-\dot{\sigma}^2
										 P(\sigma(t)) \\
									& = (1 - \dot\sigma^2)
\big\langle \partial_p H\big(Q(t),P(t)\big), P(\sigma(t))
									\big\rangle,
            \end{align*}
						where we used \eqref{eq:Psigma} in the second line. Since $\big\langle \partial_p H\big(Q(t),P(t)\big), P(\sigma(t)) \ne 0$ on a supercritical energy level, we get
								\[
								  \dot\sigma^2(t) \equiv  1.  
								\]
								Since $\dot\sigma$ is continuous, $\dot\sigma \equiv \pm 1$. The positive case imply $\pi \circ \theta|_{K_0}$ and $\pi \circ \theta|_{K_1}$ have the same parametrization, the negative case imply they are time reversals of each other.
     \end{proof}
    
		 \begin{proof}[Proof of Proposition \ref{prop-stronglyneat}]
				 Suppose that $\theta = (Q(t), P(t))$ is not strongly neat then it must admit a tagged segment, so there exists two intervals $[t_0,t_1],[\tau_0,\tau_1] \in \mathbb{R}/T\mathbb{Z}$ for which the claims of Proposition \ref{tagged_interval} holds. If the orbit segments coincide, then $\theta|_{[t_0, t_1]}$ and $\theta|_{[\tau_0, \tau_1]}$ are two different normal lifts of the same projection, which contradicts with Corollary \ref{cor:unique-lift} in the Appendix. Therefore, the orbits segments are time-reversals:
				 \[
						 Q(t + t_0) = Q(-t + \tau_1), \quad t \in [0, t_1 - t_0].   
				 \]
				 Since $H$ is reversible, the curve
				 \[
						 (Q(-t + \tau_1), - P(-t + \tau_1))  
				 \]
				 is also an orbit. Apply Corollary \ref{cor:unique-lift} again, we get
				 \[
				   P(t + t_0) =  - P(-t + \tau_1)),   \quad t \in [0, t_1 - t_0].   
				 \]
				 This implies the orbits $\theta(t)$ coincide with its time-reversal, i.e. $\theta$ is reversible. Since there are no reversible orbit in this energy level, we have a contradiction.
    \end{proof}
      
		\subsection{Proof of main theorems}

		We are now ready to prove the main theorems.

		\begin{proof}[Proof of Theorem \ref{main_thm}]
		For this proof, we choose the set $\Upsilon$ to be the set of symplectic matrices that does not admit a root of unity as eigenvalue. Define
		\[
		\begin{aligned}
				& \mP^1_* = \{(s_0, x_0, T, u) \in \mP^1 : \\
				& \qquad \text{ the linearized Poincar\'e map does not admit 1 as an eigenvalue} \}, \\
				& \mP^2_* = \{ (s_0, x_0, T, u) \in \mP^2: \\
				& \qquad \text{ the linearized Poincar\'e map does not take value in $\Upsilon$} \}.
		\end{aligned}
		\]
		The theorem follows if we can show $\Pi(\mP^1 \setminus \mP^1_*)$ is a subset of meager $F_\sigma$ set, i.e, a meager set.

		Define the map $R_k(s_0, x_0, T, u) = (ks_0, x_0, T, u)$ for $k \in \mathbb{N}$, then $R_k$ is a homeomorphism from $\mP^2$ to its image. We have
		\[
		  \mP^1 = \bigcup_k R_k(\mP^2), \quad   
		  \mP^1_* = \bigcup_k R_k(\mP^2_*).
		\]
		By Propsitions \ref{prop:p2p3} and  \ref{prop:p3p5}, $\Pi(\mP^2 \setminus \mP^5(\Upsilon))$ is a meager $F_\sigma$ set. Since $\mP^2_* \subset \mP^2 \setminus \mP^5(\Upsilon)$, $\Pi(\mP^2 \setminus \mP^2_*)$ is meager, therefore so is $\Pi(\mP^1 \setminus \mP^1_*)$. 
\end{proof}

\begin{proof}[Proof of Theorem \ref{sub-Riemannian}]
		Given an energy $k > 0$, let $\star \in \{\qd, \rf\}$, 
\[
  \cM_k := \{U \in  C^\infty(M) \mid k - \max_q U(q) > 0\}, 
\]
and define the Maupertuis map
\[
  F: \cK^\star_\cD \times \cM_k \to \cK^\star_\cD, \quad	(K, U) \mapsto  \frac{K}{k - U} \in \cK_\cD^{\star}.
\]
Note that for a fixed $U_0$, the mapping $K \mapsto F(K, U_0)$ is a homeomorphism.

By Theorem \ref{main_thm} there exists a dense $G_\delta$ set $\cG \subset \cM_k$ such that $(K, U)$ satisfies the bumpy metric theorem for $\cD-$regular periodic orbits. We will show $F(\cG)$ contains a dense $G_\delta$ subset in $\cK^\star_\cD$.  

Call the complement of a meager set comeager. We invoke the Kuratowski-Ulam theorem (see \cite{Kec95}, Theorem 8.41), which says if $A \subset X \times Y$ is comeager, where both $X, Y$ are second countable, then for a comeager subset of $x \in X$, the set $\{y \mid (x, y) \in A\}$ is comeager. Applied to the set $\cG \subset \cK^\star_\cD \times \cM_k$, we obtain at least for some $U_0 \in \cM_k$, the set
\[
		\{K \mid (K, U_0) \in \cG\}  
\]
is comeager in $\cK^\star_\cD$. Since $F(\cdot, U_0)$ is a homeomorphism, 
\[
  F\left( \{(K, U_0) \mid (K, U_0) \in \cG\}  \right)  
\]
is comeager in $\cK^\star_\cD$. 

Note that the Hamiltonian flow of $H = K + U$ on energy level $k$ is a time change of the Hamiltonian flow of $F(K, U)$ on energy level $1$. We conclude that conclusion of Theorem \ref{main_thm} holds for all $\{F(K, U) \mid (K, U_0) \in \cG\}$, and the latter is a comeager set. 
\end{proof}

 \section{The \mane control problem}
 \label{sec:mane-control}

 We prove Theorem \ref{thm:controllable} in this section. First, let us denote by $F_{ij}$ the matrix in $\mM(d)$ with $(i,j)$ entry equals to $1$ and the rest equal $0$. Then $F_{ij}$ is the standard basis of $\mM(d)$. Let 
 \[
		 E_{ij} = F_{ij} + F_{ji}  
 \]
 which forms a basis of $\mS(d)$.

 Let $\bn: [0, \delta] \to \R^d$  be the unit vector in the direction of $\mD^\perp_{te_1}$ in the normal coordinates. By Lemma \ref{lem:n0}, $\bn$ satisfies 
\begin{equation}  \label{eq:n0-proof}
   \|\bn(t)\| = 1, \quad \bn(0) = e_d, \quad \dot\bn(0) \ne 0. 
\end{equation}
We wish to prove that a generic smooth curve $A \in \cA_\delta(\bn)$ is \mane controllable, that is, the end-point mapping for following control problem 
\begin{equation}  \label{eq:control}
		\dot{X}_w(t)=
		\begin{bmatrix}
				0  & A(t) \\ 
				0  & 0 
		\end{bmatrix}X_w(t)+ 
		\sum_{1 \le i \le j \le d } 
		w_{ij}(t)  
		\begin{bmatrix}
				0 & 0 \\ 
				E_{ij} & 0
		\end{bmatrix} X_w(t), \quad X_w(0)=I.
\end{equation}
is a submersion at $w=0$. See Definitions \ref{def:cA} and \ref{def:mane-control}.

 \subsection{The bracket generating condition}

 The proposition that follows has been proved by Rifford and Ruggiero in \cite{RR12}, and it provides a sufficient condition for controllability of first order of the end-point mapping $\mE.$ To see a proof, check Section 2 of \cite{RR12}, Section 2.5 of \cite{Lazrag2014ControlTA}, or Section 2.1 of \cite{Aslanithesis}.  
 \begin{proposition}\label{prop:bracket}
 Let $\mE$ be the end-point mapping (see Definition \ref{def:mane-control}) associated with the control problem given in equation  \eqref{eq:control}. Define 
 \begin{align}
		 B^1_{ij}(t)&:= 
		 \begin{bmatrix}
				 0 & 0 \\
				 E_{ij} & 0
		 \end{bmatrix}, \\
		 B^{\ell+1}_{ij}(t)&:=[Y(t),B^{\ell}_{ij}(t)]+\dot{B}^{\ell}_{ij}, \label{eq_B}
 \end{align}
 where $[A, B] = AB - BA$, and $Y(t) = \bmat{0 & A(t) \\ 0 & 0}$. Then 
 $$d\mE(0)\big(L^1([0,\delta];\mathbb{R}^{2d})\big)=T_{X(\delta)}Sp(2d,\mathbb{R})$$ 
 if there exists $\bar{t} \in [0,\delta]$ such that
 $$
 span\big\{\{B_{ij}^{\ell}(\bar{t}) \mid i,j \in \{1,2, \hdots, d\}, \hspace{1.2 mm}\ell \in \mathbb{N}\big\}=T_ISp(2d,\mathbb{R}).
		 $$
 \end{proposition}

 Since $\frac{d}{dt}\|\bn(t)\|^2 = 0$, we have $\dot\bn(0) \cdot \bn(0) = \dot\bn(0) \cdot e_d = 0$. Therefore
 \begin{equation}  \label{eq:v}
		 \dot\bn(0) = \bmat{\bar{v} \\ 0} = \bmat{\bar{v}_1 \\ \vdots \\ \bar{v}_{d-1} \\ 0}, 
		 \quad \bar{v} \ne 0.
 \end{equation}

 Let $\cO(d-1, 1) \subset \cO(d)$ denote the space of orthogonal matrices preserving $e_d$. In other words, if $G \in \cO(d-1, 1)$ then
 \[
		 G = \bmat{\bar{G} & 0 \\ 0 & 1}, \quad \bar{G} \in \cO(d-1). 
 \]
 
 Note that if $\bn$ satisfies \eqref{eq:n0-proof}, and $A \in \cA_\delta(\bn)$, then $G \bn$ also satisfies \eqref{eq:n0-proof}, and  we have 
    $$\Ad_G(A) := G A G^T \in \cA_\delta(G\bn).$$

\begin{proposition}\label{prop:G-conjugation}
 The curve $A \in \cA_\delta(\bn)$ is \mane controllable if and only if $\Ad_G(A) \in \cA_\delta(G \bn)$ is \mane controllable. 
\end{proposition}

       \begin{proof}
        The differential of the end-point mapping $\mE$ at $w = 0$ can be written as 
        \begin{equation}
           d\mE(0)(v)=\sum_{i \le j}  X(\delta)\int_0^\delta v_{ij}(t)X^{-1}(t)
                \begin{bmatrix}
                    0 & 0 \\ 
										E_{ij} & 0 
                \end{bmatrix}X(t) dt.
        \end{equation}
            Where $X(t)$ is the solution of the homogeneous initial value problem with respect to (\ref{eq:control}), that is 
            $$\dot{X}(t)=
            \begin{bmatrix}
                0 & A(t) \\ 
                0 &  0
            \end{bmatrix}X(t), \quad X(0)=I. $$
            Let $G \in \mO(d-1,1)$ be given. Consider the following control problem 
             \begin{align*}
                         \dot{X}_w(t)=
                    &\begin{bmatrix}
                        G & 0 \\
                        0 & G
                    \end{bmatrix}     
                \begin{bmatrix}
                    0  & A(t) \\ 
                    0 & 0 
                \end{bmatrix}
                    \begin{bmatrix}
                        G^T & 0 \\ 
                        0 & G^T
                    \end{bmatrix}X_w(t)+ \\
                    &\sum_{i \le j} 
                                    w_{ij}(t)  
                                        \begin{bmatrix}
                                            0 & 0 \\ 
																						G E_{ij} G^T & 0
                                        \end{bmatrix} X_w(t), \quad X_w(0)=I. \numberthis \label{eq_control_G}
                \end{align*}
								Note that $G E_{ij} G^T$ is a basis for $\mS(d)$. Let $\tilde{\mE}$ be the end-point mapping associated with (\ref{eq_control_G}), and $\mG=\begin{bmatrix}
                G & 0 \\
                0 & G
           \end{bmatrix}$. We have 
               \begin{align*}
                    d \tilde{\mE}(0)(v)&=\sum_{i \le j}  \mG X(\delta)\mG^T\int_0^\delta v_{ij}(t)\mG X^{-1}(t)\mG^T
                \begin{bmatrix}
                    0 & 0 \\ 
										E_{ij} & 0 
                \end{bmatrix}\mG X(t) \mG^T dt, \\ 
                    &=\mG \bigg(\sum_{i \le j}   X(\delta)\mG^T \mG \int_0^\delta v_{ij}(t) X^{-1}(t)
                \begin{bmatrix}
                    0 & 0 \\ 
										G E_{ij} G^T & 0 
                \end{bmatrix} X(t) dt \bigg)\mG^T \\ 
                &=\mG \bigg(\sum_{i \le j}   X(\delta)\int_0^\delta v_{ij}(t) X^{-1}(t)
                \begin{bmatrix}
                    0 & 0 \\ 
										G E_{ij} G^T & 0 
                \end{bmatrix} X(t) dt \bigg)\mG^T  \\ 
                &=\mG d\mE(0)(\tilde{v}) \mG^T, 
               \end{align*}
               where $\tilde{v}$ is given by  
                $$
                    \sum_{i \le j} v_{ij}(t)G E_{ij} G^T=\sum_{i \le j} \tilde{v}_{ij}(t) E_{ij}.
                $$
            So, $d \mE(0)$ is surjective if and if $d \tilde{\mE}(0)$ is surjective. 
        \end{proof}

Here is our main technical result.
\begin{theorem}\label{thm:bracket-gen}
		For a given fixed $\bn(t)$ satisfying \eqref{eq:n0-proof}, then there exists an open and dense subset $\cV$ of $\cA_\delta(\bn)$, such that for all $A \in \cV$,
		\begin{equation}  \label{eq:span}
				\begin{aligned}
& \Span\{B_{ij}^1(0), \cdots, B_{ij}^5(0) \st 1 \le i, j \le d, \, \} = Sp(2d,\mathbb{R}) \\
& \quad = 
\left\{ 
						\bmat{M & S \\ T & - M^T} \mid M \in \cM(d), \quad  S, T \in \mS(d)
				\right\}.
				\end{aligned}
		\end{equation}

\end{theorem}

It's clear that Theorem \ref{thm:controllable} follows from Proposition \ref{prop:bracket} and Theorem \ref{thm:bracket-gen}.

\subsection{Reducing to a finite parameter family of perturbations}

For $l \in \mathbb{N}$ and, let $\bJ_l$ denote the $l$-jet space for smooth functions $\R \to \mS(d)$ at $t = 0$ such that $A(0) = \bmat{I_{d-1} & 0 \\ 0 & 0 }$. Given a curve $\bn(t)$, let $\bJ_l(\bn)$ be the space of all the $l$-jets coming from $A \in \cA_\delta(\bn)$. This space depends only on finitely many linear conditions whose coefficients are up to $l$ derivatives of $\bn$ at $t = 0$, hence it is a linear subspace of $\bJ_l$. If the germ of a smooth enough curve $A(t)$ is a representative of an element in $\bJ_l(\bn)$, we will simply write $A(t) \in \bJ_l(\bn)$.

The following proposition implies Theorem \ref{thm:bracket-gen}.
\begin{proposition}\label{prop:span-jet}
		Let $l \ge 3$, then there exists an open and dense subset $\cV_J \subset \bJ_l(\bn)$ such that for all $A(t) \in \cV_J$, \eqref{eq:span} holds. 	
\end{proposition}

A function $A: (-\delta, \delta) \to \mS(d)$ is real analytic if it admits a holomorphic extension to a complex neighborhood of $(-\delta, \delta) \subset \C$. Clearly, the jet space $\bJ_l$ can be identified with an appropriate equivalent class of analytic germs. 

\begin{proposition}[\cite{Kat95}, section 6.2]
		Let $A_0 : (-\delta, \delta) \to \mS(d)$ be real analytic, then there exists analytic family of orthognal matrices $P(t)$ and diagonal matrices $\Lambda_0(t)$ such that
		\begin{equation}  \label{eq:orthogonal-decomp}
				A_0(t) = P(t) \Lambda_0(t) P^T(t).
		\end{equation}
\end{proposition}

\begin{corollary}\label{cor-4.1}
		Suppose $A_0(t)$ is real analytic and $A_0(t) \in \bJ_l(\bn)$. Then there exists real analytic curves $P(t)$ and $\Lambda(t)$ such that \eqref{eq:orthogonal-decomp} holds, and
		\begin{equation}  \label{eq:P0}
				\Lambda_0(t) = \bmat{D_0(t) & 0 \\ 0 & 0} = \bmat{\diag\{\lambda_1(t), \cdots, \lambda_{d-1}(t)\} & 0 \\ 0 & 0}, 
				\quad
				P(t) = \bmat{* & \bn(t)} + o_l(t),
		\end{equation}
		where $o_l(t)$ denote a function that is order $l$-flat at $t = 0$.
\end{corollary}
\begin{proof}
		In the decomposition \eqref{eq:orthogonal-decomp}, there is a unique column of $P(t)$ corresponding to the unit null eigenvector of $A(t)$, which is tangent to $\bn(t)$ at order $l$ (recall that we only work in the $l$ jet space). It suffices to use a permutation matrix to move $\bn(t)$ to the last column.
\end{proof}

Let $A_0$ be the real analytic curve representing $A$ in the jet space $\bJ_l(\bn)$ satisfying \eqref{eq:orthogonal-decomp} and \eqref{eq:P0}. Define the following parametric family
\begin{equation}  \label{eq:A-param}
		\begin{aligned}
				\underline{A}(G, \mu, \alpha)
& = P(t) G \left( \Lambda_0(t) + t \Lambda_1^\mu + \frac12 t^2 A_2(\alpha) \right)
G^T P^T(t), \quad G \in \cO(d-1, 1), \\
 & \mu = (\mu_1, \cdots, \mu_{d-1}) \in \R^{d-1}, 
\quad \alpha = (\alpha_{ij}) \in \mS(d-1), 
		\end{aligned}
\end{equation}
where
\begin{equation}  \label{eq:parama}
\begin{aligned}
    \Lambda_1^\mu &= 	
		\begin{bmatrix}
				\mu_1 - \dot\lambda_1(0) & & & \\
																 & \ddots & & \\
																 & & \mu_{d-1} - \dot\lambda_1(0) & \\
																 & & & 0
		\end{bmatrix}, \\
		 A_2(\alpha) &= \frac12 t^2 \bmat{(\alpha_{ij})_{1 \le i, j \le d-1} & 0 \\ 0 & 0}.
\end{aligned}
\end{equation} 
We check that $\underline{A}(G, \mu, \alpha) \in \bJ^l(\bn)$. Apply the adjoint action $\Ad_G$ to $\underline{A}(G, \mu, \alpha)$, we obtain a family
\[
		\begin{aligned}
				A_{G, \mu, \alpha}  := & \Ad_{G^T} (\underline{A}(G, \mu, \alpha)) \in \bJ_l(G\bn), \\
				\quad A_{G, \mu, \alpha}(t)  = & P_G(t)  \left( \Lambda_0(t) + t \Lambda_1^\mu + \frac12 t^2 A_2(\alpha) \right) P_G(t),
		\end{aligned} 
\]
where $P_G(t) = G P(t)G^T$. Note that
\begin{equation}  \label{eq:PG}
		P_G(0) = I_d, \quad P_G(t) = \bmat{* & G\bn(t)}. 
\end{equation}

We will prove the following proposition in the next section.
\begin{proposition}\label{prop:param-controllable}
		For a given jet $A_0 \in \bJ^l(\bn)$, there is an open and dense subset $\cV \in \cO(d-1,1) \times \R^{d-1}$ and a function $\alpha = \alpha(G, \mu) : \cV \to \mS(d-1)$, such that the curve
		\[
				A_{G, \mu, \alpha(G, \mu)}(t)
		\]
		verifies \eqref{eq:span}.
\end{proposition}

Let us show how Proposition \ref{prop:param-controllable} implies Propsition \ref{prop:span-jet}, which then imply the main result of this section Theorem \ref{thm:bracket-gen}.
\begin{proof}[Proof of Proposition \ref{prop:span-jet} assuming Proposition \ref{prop:param-controllable}]
		For the parameter family $A_{G, \mu, \alpha}$, let us denote by
		\[
				B_{ij}^k(t; G, \mu, \alpha)
		\]
		the matrices $B_{i,j}^k(t)$ defined by $A_{\mu, \alpha}$. We will show that \eqref{eq:span} holds for $A_{G, \mu, \alpha}$ for an open and dense subset of $(G, \mu, \alpha)$.

		Consider the mapping
		\[
				\Phi_{G, \mu, \alpha} \st (a_{ij}^1, \cdots, a_{ij}^5) \mapsto
				\sum_{1 \le i \le j \le d, \, 1 \le k \le 5} a_{ij}^k B_{ij}^k(0; \mu, \alpha),
				\quad (\mS(d))^5 \to Sp(2d,\mathbb{R}).
		\]
		The condition \eqref{eq:span} is equivalent to $\Phi_{G, \mu, \alpha}$ having full rank, which is equivalent to the following map
		\[
				\Phi_{G, \mu, \alpha}^* \Phi_{G, \mu, \alpha} \st Sp(2d,\mathbb{R}) \to Sp(2d,\mathbb{R})
		\]
		being an isomorphism. Here $\Phi_{G, \mu, \alpha}^*$ is the adjoint of $\Phi_{G, \mu, \alpha}$ with respect to the standard inner product in Euclidean spaces, after identifying $Sp(2d,\mathbb{R})$ with $\R^{2d^2 + d}$. 

		For every fixed $G$, the coefficients of $\Phi_{G, \mu, \alpha}$ are polynomials in $\mu, \alpha$, so is 
        $$p(\mu, \alpha) = \det \Phi_{G, \mu, \alpha}^* \Phi_{G, \mu, \alpha}.$$ 
        Since the roots of a nontrivial polynomial is a nowhere dense set, it suffices to show that $p(\mu, \alpha) \ne 0$ for some particular choices of $(\mu, \alpha)$.

		Proposition \ref{prop:param-controllable} implies that for a dense subset of $G \in \cO(d-1, 1)$, there exists some $\mu$ such that $(G, \mu) \in \cV$, and $p(\mu, \alpha(G, \mu)) \ne 0$ for $\alpha = \alpha(G, \mu)$, verifying the claim.

		We have established \eqref{eq:span} for $A_{G, \mu, \alpha} \in \bJ_l(G\bn)$, for a generic $G$ and generic $(\mu, \alpha)$ depending on $G$. Since $A_{G, \mu, \alpha} = \Ad_G \underline{A}(G, \mu, \alpha)$, we  apply Proposition \ref{prop:G-conjugation} to get \eqref{eq:span} holds for a generic $\underline{A}(G, \mu, \alpha) \in \bJ_l(\bn)$. In particular, one can choose $\underline{A}(G, \mu, \alpha)$ arbitrarily close to $A_0$ in $\bJ_l(\bn)$. This proves the density of property \eqref{eq:span}. Since this condition is obviously open, the proposition follows.
\end{proof}

\subsection{Choice of the family $\alpha(G, \mu)$}

Let us now clarify the choice of the function $\alpha(G, \mu)$. If $P(t)$ is an curve of orthogonal matrices, differentiating $P P^T = I_d$ shows
\[
	P \dot{P}^T + \dot{P} P^T = 0. 
\]
If $P(0) = I_d$, then $\dot{P}(0)$ is skew-symmetric.

\begin{lemma}\label{lem:der-A}
	Let 
	\[
		A(t) = P(t) \Lambda(t) P^T(t), \quad
		P(0) = I_d, \quad
		\dot{P}(0)
		= \bmat{Q & v \\ - v^T & 0}, 
		\quad
		\ddot{P}(0) 
		= \bmat{* & w}, 
	\]
	\[
		\Lambda(0) =  \bmat{D(0) & 0 \\ 0 & 0}, \quad
		D(0) = I_{d-1},
	\]
	where $D(t)$ is a diagonal matrix.
	Then
	\begin{enumerate}[$(1)$]
		\item 
			\[
				\dot{A}(0) = \bmat{\dot{D}(0) & - v \\ -v^T & 0}.
			\]
		\item 
			\begin{equation}  \label{eq:ddotA-1st}
				\begin{aligned}
					\ddot{A}(0) & = - \bmat{ 0 & w} - \bmat{0 \\  w^T} - 2 \bmat{ vv^T & 0 \\ 0 & 0}  \\
											& \quad + 2 \bmat{Q \dot{D}(0) - \dot{D}(0) Q & - \dot{D}(0) v \\ - v^T \dot{D}(0) & 0}
											+ \ddot{\Lambda}(0). 
				\end{aligned}
			\end{equation}
	\end{enumerate}
\end{lemma}
\begin{proof}
	(1) We have
	\[
		\begin{aligned}
			\dot{A}(0) & = P(0) \dot{\Lambda}(0) P^T(0)  + \dot{P}(0) \Lambda(0) P^T(0) 
			+ P(0) \Lambda(0) \dot{P}^T(0) \\
								 & = \dot{\Lambda}(0) + \dot{P}(0) \Lambda(0) - \Lambda(0) \dot{P}(0)  \\
								 & = \bmat{\dot{D}(0) & 0 \\ 0 & 0} 
			+ \bmat{Q & v \\ -v^T & 0} \bmat{I_{d-1} & 0 \\ 0 & 0} 
			-  \bmat{I_{d-1} & 0 \\ 0 & 0} \bmat{Q & v \\ -v^T & 0} \\
								 & = \bmat{\dot{D}(0) & - v \\ -v^T & 0},
		\end{aligned}
	\]
	where we used $\dot{P}^T(0) = - \dot{P}(0)$ in the second line.

	(2) Taking the second derivative of the equation $P(t)P^T(t) = I$ at $t = 0$, we get
	\[
		(\ddot{P} + 2 \dot{P}\dot{P}^T + \ddot{P})(0) = 0.
	\]
	Recall that $F_{dd}$ denote the matrix with $(d, d)$ entry $1$ and the other entries $0$, then $I_d = \Lambda(0) + F_{dd}$. Then
	\[
		\begin{aligned}
			0 & = 
			(\ddot{P} + 2 \dot{P}\dot{P}^T + \ddot{P})(0)  \\
				& = \ddot{P}(0) (\Lambda(0) + F_{dd}) + 2 \dot{P}(0) (\Lambda(0) + F_{dd}) \dot{P}^T(0) + (\Lambda(0) + F_{dd})\ddot{P}^T(0),
		\end{aligned}
	\]
	hence
	\begin{equation}  \label{eq:ddotP-identity}
		\begin{aligned}
& \ddot{P}(0) \Lambda(0) + 2 \dot{P}(0) \Lambda(0) \dot{P}^T(0) + \Lambda(0) \ddot{P}^T(0)  \\
& = - \ddot{P}(0) F_{dd} - 2 \dot{P}(0) F_{dd} \dot{P}^T(0) - F_{dd} \ddot{P}^T(0)  \\ 
& = - \bmat{ * & w}  \bmat{0 & 0 \\ 0 & 1} - \bmat{Q & v \\ -v^T & 0} \bmat{0 & 0 \\ 0 & 1} \bmat{Q & - v \\ v^T & 0} - \bmat{0 & 0 \\ 0 & 1} \bmat{* \\ w^T} \\
& = - \bmat{ 0 & w} - \bmat{0 \\ w^T} - 2 \bmat{ vv^T & 0 \\ 0 & 0}.
		\end{aligned}
	\end{equation}
	Since
	\[
		\begin{aligned}
			\ddot{A}(0)
& =  \ddot{P}(0) \Lambda(0) + 2 \dot{P}(0) \Lambda(0) \dot{P}^T(0) + \Lambda(0) \ddot{P}^T(0)  \\
& \quad + 2\left( \dot{P}(0) \dot{\Lambda}(0) - \dot{\Lambda}(0) \dot{P}(0)\right)
+ \ddot{\Lambda}(0)
		\end{aligned}
	\]
	and
	\[
		\dot{P}(0) \dot{\Lambda}(0) - \dot{\Lambda}(0) \dot{P}(0)
		=  \bmat{Q \dot{D}(0) - \dot{D}(0) Q & - \dot{D}(0) v \\ - v^T \dot{D}(0) & 0},
	\]
	the lemma follows.
\end{proof}

Since $\|\bn(t)\|^2 = 1$, we have
\[
		\dot{\bn}(t) \cdot \bn(t) = 0, \quad \ddot{\bn}(t) \cdot \bn(t) + \dot{\bn}(t)\cdot \dot{\bn}(t) = 0. 
\]
It follows that
\begin{equation}  \label{eq:bn-notation}
		\dot{\bn}(0) = \bmat{\bar{v} \\ 0}, \quad \ddot{\bn}(0) = \bmat{\bar{w} \\ \hat{w}},   \quad
		\hat{w} =  - \|\bar{v}\|^2. 
\end{equation}

\begin{proposition}\label{prop:alpha-mu}
		With the notations in \eqref{eq:bn-notation}, furthur denote
	\[
		\Gamma = \diag\{\mu_1, \cdots, \mu_{d-1}\}, \quad
		G = \bmat{\bar{G} & 0 \\ 0 & 1}
	\]
	For the family \eqref{eq:A-param}, we can choose $\alpha = \alpha(G, \mu)$ so that the family $A_{G, \mu, \alpha}$ satisfies
	\begin{equation}  \label{eq:dotA-mu}
		\dot{A}_{G, \mu, \alpha}(0) = 
		\begin{bmatrix}
			\mu_1 & & & \\
						& \ddots & & - \bar{G} \bar{v} \\
						& & \mu_{d-1} & \\
						& - \bar{v}^T \bar{G}^T & & 0
		\end{bmatrix}.
	\end{equation}
	and
	\begin{equation}  \label{eq:ddot-A-mu}
			\ddot{A}_{G, \mu, \alpha}(0) = - \bmat{ 0 &  \bar{G}\bar{w} \\ 0 & \hat{w}} - \bmat{0 & 0 \\ \bar{w}^T \bar{G}^T & \hat{w}}
		- 2 
		\begin{bmatrix}
			0 & \Gamma \bar{G} \bar{v}  \\
			\bar{v}^T \bar{G}^T \Gamma & 0
		\end{bmatrix}.
	\end{equation}
\end{proposition}

\begin{proof}
	Let us write
	\[
		A_{G, \mu, \alpha} = A_{G, \mu, 0}(t) + \frac12 t^2 \bmat{(\alpha_{ij}) & 0 \\ 0 & 0},
	\]
	then
	\[
		A_{G, \mu, 0} = P_G(t) \Lambda(t) P_G^T(t), \quad
		\dot{\Lambda}(0) = \diag\{\mu_1, \cdots, \mu_{d-1}, 0\}, \quad
		\ddot{\Lambda}(0) = \ddot \Lambda_0(0) =: \bmat{\ddot{D}_0(0) & 0 \\ 0 & 0}.
	\]
	By Lemma \ref{lem:der-A}, and \eqref{eq:PG}, 
	\[
		\dot{A}_{\alpha, \mu}(0) = \dot{A}_{\alpha, 0} = \bmat{\Gamma & - \bar{G} v \\ -v^T \bar{G}^T & 0}
	\]
	as required. Moreover,
	\[
		\ddot{A}_{\mu, \alpha}(0) = \ddot{A}_{\mu, 0}(0) + \bmat{(\alpha_{ij}) & 0 \\ 0 & 0}.
	\]
	Apply \eqref{eq:ddotA-1st}, it suffices to take
	\begin{equation}  \label{eq:alpha}
			(\alpha_{ij}(G, \mu)) = 2 G vv^T G^T -2 (Q \Gamma - \Gamma Q) - \ddot{D}_0(0)
	\end{equation}
	to get \eqref{eq:ddot-A-mu}.
\end{proof}

\subsection{Choosing the generating set}
We prove Proposition \ref{prop:param-controllable} in this section. Let's start with calculations of $B_{ij}^k$.
\[
	\begin{aligned}
		B_{ij}^1 & = \bmat{0 & 0 \\ E_{ij} & 0},  \\
		B_{ij}^2 & = \bmat{ A E_{ij} & 0 \\ 0 &  - E_{ij} A} \\
		B_{ij}^3 & = \bmat{0 & - 2 A E_{ij} A \\ 0 & 0}
		+ \bmat{\dot{A} E_{ij} & 0 \\ 0 & - E_{ij} \dot{A}} \\
		B_{ij}^4 & = \bmat{0 & - 3 (\dot{A} E_{ij} A + A E_{ij} \dot{A}) \\ 0 & 0} 
		+ \bmat{\ddot{A}E_{ij} & 0 \\ 0 & - E_{ij} \ddot{A}} \\  
		B_{ij}^5 & = \bmat{* & -4(\ddot{A} E_{ij} A + A E_{ij} \ddot{A}) - 6 \dot{A} E_{ij} \dot{A} \\ * & *}. 
	\end{aligned} 
\]
In the unwritten blocks of $B_{ij}^5$ up to third derivative of $A$ is involved.

We will choose the family of matrices
\begin{equation}  \label{eq:A-G-mu}
	A(t) = A_{G, \mu, \alpha(G, \mu)}
\end{equation}
which satisfies \eqref{eq:dotA-mu} and \eqref{eq:ddot-A-mu}. Setting
\begin{equation}  \label{eq:vw}
	v = v(G) = \bar{G}\bar{v}, \quad
	w = w(G, \mu) = - \bmat{\bar{G} \bar{w} \\ \hat{w}} - 2 \bmat{\Gamma v(G) \\ 0}, 
\end{equation}
we rewrite \eqref{eq:dotA-mu} and \eqref{eq:ddot-A-mu} as
\begin{equation}  \label{eq:dotA}
	\dot{A}(0)
	= \bmat{\Gamma & - v \\ - v^T & 0}
\end{equation}
and 
\begin{equation}  \label{eq:ddotA}
	\ddot{A}(0) = \bmat{0 & w} + \bmat{0 \\ w^T}.
\end{equation}
For most of the calculations below, $v$ and $w$ will be considered constant parameters. We will invoke their dependence on $G, \mu$ only near the end of the proof.

We will start by choosing a particular generating set from $B_{ij}^2, B_{ij}^3, B_{ij}^4$ that makes the remaining proof of Proposition \ref{prop:param-controllable} much simpler.

Let us denote
\begin{equation}  \label{eq:J1}
	J_1 = \{(i, j) \mid 1 \le i \le d-1, \, i \le j \le d\} = \{(i, j) \st 1 \le i \le j \le d\} \setminus \{(d, d)\}, 
\end{equation}
\begin{equation}  \label{eq:J2}
	J_2 = \{(i, j) \mid 1 \le i \le j \le d-1\},	
\end{equation}
and
\begin{equation}  \label{eq:S-star}
	\mS^*(d) = \{S \in \mS(d) \mid (S)_{dd} = 0\}	
	= \Span\{E_{ij} \mid (i,j) \in J_1\}.
\end{equation}

\begin{proposition}\label{prop:B234-gen}
	For the family $A(t) = A_{G, \mu, \alpha(G, \mu)}$ as defined, there is an open and dense set $\cV \subset \cO(d-1, 1) \times \R^{d-1}$ such that for $A_{\mu, G, \alpha(\mu, G)}$ with $(G, \mu) \in \cV$, we have
	\begin{equation}  \label{eq:span-B234}
		\Span
		\left\{ 
			\begin{aligned}
& B_{ij}^2 & (i,j) \in J_1 \\
& B_{ij}^3 & (i,j) \in J_2 \\
& B_{ij}^4 & (i,j) \in J_1 \\
& \sum_{i = 1}^{d-1} v_i B_{id}^3
			\end{aligned}
		\right\}
		= 
		\left\{ 
			\bmat{M & S \\ 0 & - M^T} \mid M \in \cM(d), \, S \in \mS^*(d)
		\right\}.
		\end{equation}
\end{proposition}
\begin{remark}
	The basis matrices are chosen for the following reason:
\begin{itemize}
 \item  
	 If we project $B_{ij}^2$, $(i, j) \in J_1$, $B_{ij}^3$, $(i,j) \in J_2$ and $\sum_{i = 1}^{d-1}v_i B_{id}^3$ to the main diagonal blocks, we obtain a basis of the space $\{\diag\{M, - M^T\} \mid  M \in \mM(d)\}$.
 \item Projection of $B_{ij}^4$, $(i, j) \in J_1$ to the upper right block form a basis of $\mS^*(d)$.
\end{itemize}
These two observation alone does not imply our proposition, as there are nontrivial interactions among the other blocks. If $d = 2$, the spanning set is $6$-dimensional, and it is possible to check by hand that for a generic choice of the vectors $\mu, v$, the given matrices form a spanning set. The general case is more technical.
\end{remark}

\begin{proposition}\label{prop:B5}
	For arbitrary choice of $\mu, G$ and the corresponding $A_{\mu, G, \alpha(\mu, G)}$, 
	the subspace $\Span\{B_{ij}^5 \st 1 \le i = j \le d -1\}$	contains at least one matrix
	\[
		\bmat{* & S \\ * & *}  
	\]
	for which $(S)_{dd} \ne 0$. 
\end{proposition}

\begin{proof}[Proof of Proposition \ref{prop:param-controllable} assuming Proposition \ref{prop:B234-gen} and \ref{prop:B5}]
  Recall that $T_I Sp(2d, \R)$ takes the block form
	\[
			\bmat{M & S_2 \\ S_1 & - M^T}, \quad M \in \cM(d), S_1, S_2 \in \mS(d).  
	\]
	Observe that $B_{ij}^1$ spans all symmetric matrices in the lower left block. Proposition \ref{prop:B234-gen} then implies the following matrix for arbitrary $M \in \cM(d)$, $S_1 \in \mS(d)$, $S_2^* \in \mS^*(d)$,
	\[
		\bmat{M & S_2^* \\ S_1 & - M^T}  
	\]
	is contained in $\Span\{B_{ij}^l \st 1 \le l \le 4\}$. Then the special matrix provided by Proposition \ref{prop:B5} together with the above matrices spans all of $Sp(2d,\mathbb{R})$.
\end{proof}

We prove Proposition \ref{prop:B234-gen} and \ref{prop:B5} in the rest of this section.

Let's denote
\begin{equation}  \label{eq:matrix-names}
	\begin{aligned} 
& \xi_{ij} = A(0) E_{ij}, \quad \eta_{ij} = \dot{A}(0) E_{ij}, \quad \zeta_{ij} = \ddot{A}(0) E_{ij}, \\
& \gamma_{ij} = A(0)E_{ij} A(0), \quad \kappa_{ij} = \dot{A}(0) E_{ij} A_{ij}(0),
	\end{aligned}  
\end{equation}
then
\[
	B_{ij}^2 = \bmat{\xi_{ij} & 0 \\ 0  & - \xi_{ij}^T }, \quad
	B_{ij}^3 = \bmat{\eta_{ij} & - 2 \gamma_{ij} \\ 0 & - \eta_{ij}^T}, \quad
	B_{ij}^4 = \bmat{\zeta_{ij} & - 3 \kappa_{ij} \\ 0 &  - \zeta_{ij}^T}.
\]

\begin{lemma}\label{lem:independence}
	Formula \eqref{eq:span-B234} holds if and only if the following equations
	\begin{equation}  \label{eq:comb-diagonal}
		\sum_{(i,j) \in J_1} \left( a_{ij} \xi_{ij} + b_{ij} \eta_{ij} + c_{ij} \zeta_{ij} \right)  = 0,
	\end{equation}
	\begin{equation}  \label{eq:comb-upper-right}
		\sum_{(i,j) \in J_1} \left( 2 b_{ij} \gamma_{ij} + 3 c_{ij} \kappa_{ij}\right) = 0, 
	\end{equation}
	\begin{equation}  \label{eq:bid}
		b_{id} = b v_i, \quad 1 \le i \le d-1	
	\end{equation}
	imply 
	\[
		a_{ij} = b_{ij} = c_{ij} = b = 0, \quad (i,j) \in J_1.  
	\]
\end{lemma}
\begin{proof}
	Since $\#J_1 = \frac{d(d+1)}{2} - 1$ and $\#J_2 = \frac{d(d-1)}{2}$, we check that the number of spanning elements in the left hand side of \eqref{eq:span-B234} is equal to the dimension of the right hand side. Therefore, it suffices to prove the spanning elements are linearly independent.

	By denoting $b_{id} = b v_i$, we have
	\[
		\sum_{J_2} b_{ij} B_{ij}^3 + b \sum_{i = 1}^{d-1} v_i B_{id}^3
		= \sum_{J_1} b_{ij} B_{ij}^3, 
	\]
	this allows us to write the linear combination of the spanning elements as
	\[
		\sum_{J_1} \left( a_{ij} B_{ij}^2 + b_{ij} B_{ij}^3 + c_{ij} B_{ij}^4\right). 
	\]
	This matrix is zero if and only if the upper left and upper right block is zero. These are precisely \eqref{eq:comb-diagonal} and \eqref{eq:comb-upper-right}.
\end{proof}

For the matrices $\xi_{ij}, \eta_{ij}, \zeta_{ij}$, we calculate their coordinates in $F_{ij}$, the standard basis of $\cM(d)$. For the matrices $\gamma_{ij}, \kappa_{ij}$ we express them in terms of $E_{ij}$, $(i, j) \in J_1$, the standard basis of $\mS^*(d)$.

\begin{lemma}\label{lem:matrix-explicit}
	We have the following calculations:
	\begin{equation}  \label{eq:xi-ij}
		\begin{aligned}
			\xi_{ij} 
			& = \begin{cases}
				F_{ij} + F_{ji} & 1 \le i \le j \le d, \\
				F_{id} & 1 \le i \le d-1, \, j = d, \\
				0 & i = j = d. 
			\end{cases}
		\end{aligned}
	\end{equation}
	\begin{equation}  \label{eq:eta-ij}
		\eta_{ij} = 
		\begin{cases}
			\mu_i F_{ij} +  \mu_j F_{ij} - v_i F_{dj} - v_j F_{di} & 1 \le i \le j \le d-1, \\
			\mu_i F_{id}  - v_i F_{dd} - \sum_{k = 1}^{d-1} v_k F_{ki} & 
			1 \le i \le d-1, \, j = d, \\
			- 2 \sum_{k = 1}^{d-1} v_k F_{kd} & i = j = d .
		\end{cases}
	\end{equation}
	\begin{equation}  \label{eq:zeta-ij}
		\zeta_{ij}
		= \begin{cases}
			w_i F_{dj} + w_j F_{di}, & 1 \le i \le j \le d-1, \\
			\sum_{k = 1}^{d-1} w_k F_{ki} + w_i F_{dd} + 2w_d F_{di}, & 1 \le i \le d-1, \, j = d, \\
			2\sum_{k = 1}^{d-1} w_k F_{kd} + 4w_d F_{dd},  & i = j = d. 
		\end{cases} 
	\end{equation}
	\begin{equation}  \label{eq:gamma-ij}
		\gamma_{ij} = 
		\begin{cases}
			E_{ij} & 1 \le i \le j \le d-1 \\
			0 & \text{otherwise}.
		\end{cases}
	\end{equation}
	\begin{equation}  \label{eq:kappa-ij}
		\begin{aligned}
			\kappa_{ij} & = 
			\begin{cases}
				(\mu_i + \mu_j) E_{ij} - v_i E_{jd} - v_j E_{id} & 1 \le i \le j \le d-1, \\
				- \sum_{k = 1}^{d-1} v_k E_{ki} & 
				1 \le i \le d-1, \, j = d, \\
				0  & i = j = d. 
			\end{cases}
		\end{aligned}
	\end{equation}
\end{lemma}
\begin{proof}
	By \eqref{eq:dotA} and \eqref{eq:ddotA}, we have
	\begin{equation}  \label{eq:dotA-ij}
		\dot{A}(0) = \sum_{i = 1}^{d-1} \mu_i F_{ii} - \sum_{i = 1}^{d-1} v_i (F_{ij} + F_{ij}), 
	\end{equation}
	\begin{equation}  \label{eq:ddotA-ij}
		\ddot{A}(0) = \sum_{k = 1}^d w_k (F_{kd} + F_{dk}).
	\end{equation}

		We have
		\begin{equation}  \label{eq:AE}
			\begin{aligned}
				\xi_{ij} = A(0) E_{ij} & = \sum_{k = 1}^{d-1} F_{kk} (F_{ij} + F_{ji})	\\
															 & = \begin{cases}
																 F_{ij} + F_{ji} & 1 \le i \le j \le d - 1, \\
																 F_{id} & 1 \le i \le d-1, \, j = d, \\
																 0 & i = j = d, 
															 \end{cases}
			\end{aligned}
		\end{equation}
		which is \eqref{eq:xi-ij}. 

		For \eqref{eq:eta-ij}, we multiply the two parts of \eqref{eq:dotA-ij} to $E_{ij}$ separately:
		\[
			\sum_{k = 1}^{d-1} \mu_k F_{kk} (F_{ij} + F_{ji}) 
			= \begin{cases}
				\mu_i F_{ij} + \mu_j F_{ji} & 1 \le i \le j \le d, \\
				\mu_i F_{id} & 1 \le i \le d-1, \, j = d, \\
				0 & i = j = d.
			\end{cases}
		\]
\begin{equation}  \label{eq:etaij-2}
			\begin{aligned}
 & - \sum_{k = 1}^{d-1} v_k  (F_{kd} + F_{dk})(F_{ij} + F_{ij}) =  - \sum_{k = 1}^{d-1} v_k F_{dk} (F_{ij} + F_{ji})
 - \sum_{k = 1}^{d-1} v_k F_{kd} (F_{ij} + F_{ji}) \\
 & = \begin{cases}
	 - v_i F_{dj} - v_j F_{di} & 1 \le i \le j \le d - 1 \\
	 - v_i F_{dd} - \sum_{k = 1}^{d-1} v_k F_{ki} & 1 \le i \le d-1, \, j = d \\
	 - 2 \sum_{k = 1}^{d-1} v_k F_{kd} & i = j = d. 
 \end{cases}
			\end{aligned}
\end{equation}
		Equation \eqref{eq:eta-ij} follows by summing the last two formulas.

		Equation \eqref{eq:zeta-ij} follows from a similar calculation to \eqref{eq:etaij-2}:
		\[
			\begin{aligned}
				\zeta_{ij}
&  = \sum_{k = 1}^{d-1} w_k(F_{kd} + F_{dk})(F_{ij} + F_{ji}) + 2w_d F_{dd} (F_{ij} + F_{ji}) \\
& = \begin{cases}
	w_i F_{dj} + w_j F_{di}, & 1 \le i \le j \le d-1, \\
	\sum_{k = 1}^{d-1} w_k F_{ki} + w_i F_{dd} + 2w_d F_{di}, & 1 \le i \le d-1, \, j = d, \\
	2\sum_{k = 1}^{d-1} w_k F_{kd} + 4w_d F_{dd} & i = j = d. 
\end{cases} 
			\end{aligned}
		\]

		Equation \eqref{eq:gamma-ij} follows from the fact that
		\[
			A(0) = \bmat{I_{d-1} & 0 \\ 0 & 0}
		\]
		and $A(0)E_{ij}A(0)$ obtained from $E_{ij}$ by setting the last row and column to $0$.

		Finally,
		\[
			\dot{A}(0) E_{ij} A(0) = \eta_{ij} \bmat{I_{d-1} & 0 \\ 0 & 0}
		\]
		is obtained by setting the last column of $\eta_{ij}$ to $0$. Explicitly,
		\[
			\dot{A}(0) E_{ij} A(0) = 
			\begin{cases}
				\mu_i F_{ij} +  \mu_j F_{ij} - v_i F_{dj} - v_j F_{di} & 1 \le i \le j \le d-1, \\
				- \sum_{k = 1}^{d-1} v_k F_{ki} & 
				1 \le i \le d-1, \, j = d, \\
				0 & i = j = d .
			\end{cases}
		\]
		Since $\kappa_{ij} = \dot{A}(0) E_{ij} A(0) + (\dot{A}(0) E_{ij} A(0))^T$, we add the above matrix to its transpose to get \eqref{eq:kappa-ij}.
	\end{proof}

	Let us note the following calculation that will be reused several times.
	\begin{lemma}\label{lem:diag-sum}
			Let $s_{ij}$, $1 \le i \le j \le d-1$, $w_i$, $x_i$, $1 \le i \le d-1$, be real numbers. Extend the definition of $s_{ij}$ to all $1 \le i, j \le d-1$ by the symmetry condition $s_{ij} = s_{ij}$. define
		\begin{equation}  \label{eq:diag-sum-sym}
			\bar{s}_i(w) = \sum_{j = 1}^i s_{ji} w_j + \sum_{j = i}^{d-1} s_{ij} w_j = 2 s_{ii} w_i + \sum_{1 \le j \le d-1, \,  j \ne i} s_{ij} w_j.
		\end{equation}
		Then	
		\begin{equation}  \label{eq:diag-sum}
			\sum_{1 \le i \le j \le d-1} s_{ij}(w_i x_j + x_i w_j)
			= \sum_{i = 1}^{d-1} \bar{s}_i(w) x_i,
		\end{equation}
	\end{lemma}
\begin{proof}
		We have
		\[
				\sum_{1 \le i \le j \le d - 1} s_{ij} w_i x_j
				= \sum_{1 \le j \le i \le d-1} s_{ji} w_j x_i
				= \sum_{i = 1}^{d-1} \sum_{j = 1}^i s_{ji} w_j x_i
		\]
		and 
		\[
				\sum_{1 \le i \le j \le d-1} s_{ij}w_j x_i 
				= \sum_{i = 1}^{d-1} \sum_{j = i}^{d-1} s_{ij} w_j x_i,
		\]
		the lemma follows immediately.
\end{proof}

	Using Lemma \ref{lem:matrix-explicit}, we now calculate \eqref{eq:comb-diagonal} and \eqref{eq:comb-upper-right}.

	\begin{lemma}\label{lem:sum-linear-comb}
			We have 

		\begin{small}\begin{equation}  \label{eq:sum-xi-eta}
			\begin{aligned}
& \sum_{J_1} a_{ij} \xi_{ij} + \sum_{J_1} b_{ij} \eta_{ij} \\
& = 
\sum_{1 \le i < j \le d-1} (a_{ij} + \mu_i b_{ij} - b_{jd} v_i) F_{ij} 
+ \sum_{1 \le i < j \le d-1} (a_{ij} + \mu_j b_{ij} - b_{id} v_j) F_{ji}  \\
& \quad 
+ \sum_{i = 1}^{d-1} (2a_{ii} + 2\mu_{ii} b_{ii} - b_{id} v_i ) F_{ii}
+ \sum_{i = 1}^{d-1} (a_{id} + \mu_i b_{id}) F_{id} 
- \sum_{i = 1}^{d-1} \bar{b}_i(v) F_{di}
\\ & \quad  - \left(  \sum_{i = 1}^{d-1} b_{id}v_i \right) F_{dd} .
			\end{aligned}
		\end{equation} \end{small}

		\begin{equation}  \label{eq:sum-zeta-ij}
			\begin{aligned}
				\sum_{J_1} c_{ij} \zeta_{ij}
& = \sum_{1 \le i < j \le d-1} w_i c_{jd} F_{ij} 
+ \sum_{1 \le i < j \le d-1} w_j c_{id} F_{ji} 
+ \sum_{i = 1}^{d-1} w_i c_{id} F_{ii} \\
& \quad
+  \sum_{i = 1}^{d-1} (\bar{c}_i(w) + 2c_{id}w_d) F_{di}
+ \left( \sum_{i = 1}^{d-1} c_{id} w_i\right) F_{dd}
			\end{aligned}
		\end{equation}
		\begin{equation}  \label{eq:sum-kappa-ij}
			\begin{aligned}
				\sum_{(i,j) \in J_1} c_{ij} \kappa_{ij}  
& = \sum_{1 \le i < j \le d-1}
\left( 
	(\mu_i + \mu_j) c_{ij} - c_{id} v_j - c_{jd} v_i
\right) E_{ij}  \\
& \quad + \sum_{i = 1}^{d-1} (2\mu_i c_{ii} - c_{id} v_i) E_{ii}
- \sum_{i = 1}^{d-1} \bar{c}_i(v) E_{id}, 
			\end{aligned}
		\end{equation}
		where (see Lemma \ref{lem:diag-sum}) 
		\begin{small}
		\[
			\bar{b}_i(v) = \sum_{j = 1}^i b_{ji} v_i  + \sum_{j = i}^{d-1} b_{ij} v_j, \quad
			\bar{c}_i(v) = \sum_{j = 1}^i c_{ji} v_i  + \sum_{j = i}^{d-1} c_{ij} v_j, \quad
			\bar{c}_i(w) = \sum_{j = 1}^i c_{ji} w_i  + \sum_{j = i}^{d-1} c_{ij} w_j.
		\]\end{small}
	\end{lemma}
	\begin{proof}
		Proof of \eqref{eq:sum-xi-eta}:
		By \eqref{eq:xi-ij}, \eqref{eq:eta-ij}, we have
		\begin{small}
		    \[
			\begin{aligned}
				& \sum_{J_1} a_{ij} \xi_{ij} + \sum_{J_1} b_{ij} \eta_{ij} \\
				& = \sum_{1 \le i \le j \le d-1} a_{ij} (F_{ij} + F_{ji}) + \sum_{i = 1}^{d-1} a_{id} F_{id} \\
				& \quad 
				+ \sum_{1 \le i \le j \le d-1} b_{ij} 
				\left( \mu_i F_{ij} + \mu_j F_{ji} - v_i F_{dj} - v_j F_{di} \right) 
				+ \sum_{i = 1}^{d-1} b_{id} \left( \mu_i F_{id} - v_i F_{dd} - \sum_{k = 1}^{d-1} v_k F_{ki} \right) \\
				& = \sum_{1 \le i \le j \le d-1} a_{ij} (F_{ij} + F_{ji})
				+ \sum_{1 \le i \le j \le d-1} b_{ij} ( \mu_i F_{ij} + \mu_j F_{ji}) 
				-\sum_{i = 1}^{d-1} \sum_{j = 1}^{d-1} b_{jd} v_i F_{ij}
				\\
				& \quad 
				+ \sum_{i = 1}^{d-1} a_{id} F_{id} + \sum_{i = 1}^{d-1} b_{id} \mu_i F_{id}
				- \sum_{i = 1}^{d-1} b_{id} v_i F_{dd}
				- \sum_{1 \le i \le j \le d-1} b_{ij} ( v_i F_{dj} + v_j F_{di}). 
			\end{aligned}
		\] \end{small}
		Observe that
		\[
				\sum_{i = 1}^{d-1} \sum_{j = 1}^{d-1} b_{jd} v_i F_{ij}
				= \sum_{1 \le i < j \le d-1} b_{jd} v_i F_{ij}
				+ \sum_{1 \le i < j \le d-1} b_{id}v_j F_{ji}
				+ \sum_{i = 1}^{d-1} b_{id} v_i F_{ii},
		\]
		and (by Lemma \ref{lem:diag-sum}) $\sum_{1 \le i \le j \le d-1} b_{ij} ( v_i F_{dj} + v_j F_{di})
		= \sum_{i = 1}^{d-1} \bar{b}_i(v) F_{di}$, \eqref{eq:sum-xi-eta} follows.

		For \eqref{eq:sum-zeta-ij}, we use \eqref{eq:zeta-ij} to get
		\[
\begin{aligned}
			 & \sum_{J_1} c_{ij} \zeta_{ij}  \\
			 & = \sum_{1 \le i \le j \le d-1} c_{ij} (w_i F_{dj} + w_j F_{di})
			 + \sum_{i = 1}^{d-1} \sum_{k = 1}^{d-1} c_{id} w_k F_{ki}
			 + \sum_{i = 1}^{d-1} c_{id} w_i F_{dd} + \sum_{i = 1}^{d-1} 2c_{id} w_d F_{di} \\
			 & = \sum_{i = 1}^{d-1} (\bar{c}_i(w) + 2c_{id} w_d) F_{di} 
			 + \sum_{i = 1}^{d-1} \sum_{j = 1}^{d-1} w_i c_{jd} F_{ij}
			 + \sum_{i = 1}^{d-1} c_{id} w_i F_{dd}.
\end{aligned}
		\]
    Equation \eqref{eq:sum-zeta-ij} follows by splitting the second sum into $i < j$, $i = j$, and $i > j$ terms.

		For \eqref{eq:sum-kappa-ij}, we compute
		\[
			\begin{aligned}
&  \sum_{(i,j) \in J_1} c_{ij} \kappa_{ij} \\
& = \sum_{1 \le i \le j \le d-1}
c_{ij} \left( 
	(\mu_i + \mu_j) E_{ij} - v_i E_{jd} - v_j E_{id} 
\right) 
- \sum_{i = 1}^{d-1} c_{id} \sum_{k = 1}^{d-1} v_k E_{ki}  \\
&   = \sum_{1 \le i < j \le d-1}
c_{ij} (\mu_i + \mu_j) E_{ij} 
+ \sum_{i = 1}^{d-1} 2 \mu_i c_{ii} E_{ii}
- \sum_{1 \le i \le j \le d-1} (v_i E_{id} + v_j E_{id}) \\
& \quad
- \left( \sum_{1 \le i < k \le d-1} + \sum_{1 \le k < i \le d-1}\right) c_{id} v_k E_{ki}
- \sum_{i = 1}^d c_{id} v_i E_{ii},
			\end{aligned}
		\]
		then use
		\[
		  \left( \sum_{1 \le i < k \le d-1} + \sum_{1 \le k < i \le d-1}\right) c_{id} v_k E_{ki}  
			=
			\sum_{1 \le i < j \le d-1} \left( c_{id} v_j + c_{jd} v_i \right) E_{ij}
		\]
		and Lemma \ref{lem:diag-sum} to get \eqref{eq:sum-kappa-ij}.
	\end{proof}

	\begin{proof}[Proof of Proposition \ref{prop:B234-gen}]
		We will first assume that $v = (v_1, \cdots, v_{d-1})$ satisfies
		\begin{equation}  \label{eq:vi}
			v_i \ne 0, \quad i = 1, \cdots, d - 1.
		\end{equation}
		This is already satisfies when $d = 2$, and for $d \ge 3$, we will eventually verify \eqref{eq:vi} holds for generic choice of parameters. We will also assume that
		\[
		  \mu_i \ne \mu_j  
		\]
		which clearly  holds for a generic set of $\mu$.

		By Lemma \ref{lem:independence}, we will show that \eqref{eq:comb-diagonal}, \eqref{eq:comb-upper-right}, \eqref{eq:bid} implies all coefficients vanish. To do this, we use elimination and substitution until we obtain a linear system involving only $c_{id}$, $1 \le i \le d-1$. We then show the coefficient matrix of this equation is non-singular generically.

		From \eqref{eq:comb-diagonal}, we use \eqref{eq:sum-xi-eta}  and \eqref{eq:sum-zeta-ij} to get: 
\begin{align}
		0 & = a_{ij} + \mu_i b_{ij} - b_{jd} v_i + w_i c_{jd} & 1 \le i < j \le d-1 \label{eq:ilj} \\ 	
		0 & = a_{ij} + \mu_j b_{ij} - b_{id} v_j + w_j c_{id} & 1 \le i < j \le d-1 \label{eq:igj} \\ 	
				0 & = 2a_{ii} + 2 \mu_i b_{ii} - b_{id} v_i + w_i c_{id} & 1 \le i \le d-1  \\
				0 & = a_{id} + \mu_i b_{id} & 1 \le i \le d-1 \\
				0 & = - \bar{b}_i(v) + \bar{c}_i(w) + 2c_{id} w_d & 1 \le i \le d-1 \label{eq:barb} \\
				0 & = - \sum_{k = 1}^d  b_{kd} v_k + \sum_{k = 1}^{d-1} c_{kd} w_k
 \end{align}
		\begin{itemize}
				\item From \eqref{eq:ilj} and \eqref{eq:igj}, we get
				\[
					b_{ij} = \frac{1}{\mu_i - \mu_j} \left( 
						(b_{jd} v_i - b_{id} v_j) - (w_i c_{jd} - w_j c_{id})
					\right)
					= - \frac{1}{\mu_i - \mu_j} (w_i c_{jd} - w_j c_{id}),
				\]
				for $1 \le i < j \le d-1$. Here we used $b_{id} = b v_i$ hence $b_{id}v_j - b_{jd} v_i = 0$.

				We note that the right hand side is symmetric in $i,j$, therefore if we set $b_{ij} = b_{ij}$, the same formula still holds. 
			\item
					We apply Lemma \ref{lem:diag-sum} to get
				\[
					\bar{b}_i(v)
					= 2 b_{ii} v_i + \sum_{j \ne i} b_{ij} v_j 
					= 2 b_{ii} v_i 
					- \sum_{j \ne i} \frac{v_j}{\mu_i - \mu_j} (w_i c_{jd} - w_j c_{id})
				\]
		\item From \eqref{eq:barb}, 
				\begin{equation}  \label{eq:bii}
					\begin{aligned}
						 2 v_i b_{ii}  & =    \sum_{j \ne i} \frac{v_i}{\mu_i - \mu_j} 
						(w_i c_{jd} - w_j c_{id}) +  \bar{b}_i(v)  \\
														& = \bar{c}_i(w) +  \sum_{j \ne i} \frac{v_i}{\mu_i - \mu_j} 
														(w_i c_{jd} - w_j c_{id}) + 2c_{id} w_d.
					\end{aligned}
				\end{equation}
		\end{itemize}

		From \eqref{eq:comb-upper-right}, we use \eqref{eq:sum-kappa-ij} and \eqref{eq:gamma-ij} to get
		\begin{align}
				0 & = 3 (\mu_i + \mu_j) c_{ij} - 3 c_{id}v_j - 3 c_{jd} v_i + 2 b_{ij}, & 1 \le i < j \le d-1,
				\label{eq:c-ilj}\\
				0 & = 6\mu_i c_{ii} -  3 c_{id} + 2b_{ii} & 1 \le i \le d-1, 
				\label{eq:c-iej}\\
				0 & = \bar{c}_i(v), & 1 \le i \le d-1. \label{eq:barc}
		\end{align}
		From \eqref{eq:c-ilj}, we have for $1 \le i < j \le d-1$:
		\begin{equation}  \label{eq:cij}
			\begin{aligned}
				c_{ij} & = \frac{1}{\mu_i + \mu_j} \left( 
					c_{id} v_j + c_{jd} v_i - \tfrac23 b_{ij}
				\right)  \\
							 & =  \frac{1}{\mu_i + \mu_j} \left( 
								 c_{id} v_j + c_{jd} v_i 
							 + \frac{2}{3(\mu_i - \mu_j)} (w_i c_{jd} - w_j c_{id})	   \right)  \\
							 & = \frac{1}{\mu_i + \mu_j}(c_{id}v_j + c_{jd} v_i)
							 + \frac{2}{3(\mu_i^2 - \mu_j^2)} (w_i c_{jd} - w_j c_{id}), 
				\end{aligned}
			\end{equation}
			and from \eqref{eq:c-iej},
			\begin{equation}  \label{eq:cii}
				6\mu_i c_{ii} - 3 c_{id} v_i + 2 b_{ii} = 0. 
			\end{equation}

			Apply \eqref{eq:diag-sum-sym}, we get from \eqref{eq:barc} that
			\begin{equation}  \label{eq:cii-2}
				c_{ii} = - \frac{1}{2  v_i} \sum_{j \ne i} c_{ij} v_j,
			\end{equation}
			where $\sum_{j \ne i}$ denotes a summation in all $1 \le j \le d-1$ such that $j \ne i$.
			Apply \eqref{eq:diag-sum} to $c_{ij}$ and $w_j$, then plug in \eqref{eq:cii-2} to get
\begin{equation}  \label{eq:cibarw}
				\bar{c}_i(w) = 2 c_{ii} w_i + \sum_{j \ne i} c_{ij} w_j 
				= - \sum_{j \ne i} c_{ij} \frac{w_i v_j}{v_i}  + \sum_{j \ne i} c_{ij} w_j  
				= \sum_{j \ne i} c_{ij} \frac{v_i w_j - w_i v_j}{v_i}.
\end{equation}

			From \eqref{eq:cii} and \eqref{eq:cii-2}, we get
			\[
				 2 b_{ii} = - 6 \mu_i c_{ii} + 3 c_{id} v_i = -  \frac{3 \mu_i}{v_i} \sum_{j \ne i} c_{ij} v_j + 3 c_{id} v_i .
			\]
			Plug this into \eqref{eq:bii}, then use \eqref{eq:cibarw} to get
			\[
				\begin{aligned}
&  - 3 \mu_i \sum_{j \ne i} c_{ij} v_j + 3 c_{id} v_i^2  =  2v_i b_{ii} 
 =  \bar{c}_i(w) 
 +  \sum_{j \ne i} \frac{v_j}{\mu_i - \mu_j} (w_i c_{jd} - w_j c_{id}) + 2c_{id} w_d \\
& = \sum_{j \ne i} c_{ij} \frac{v_i w_j - w_i v_j}{v_i}
+  \sum_{j \ne i} \frac{v_j}{\mu_i - \mu_j} (w_i c_{jd} - w_j c_{id}) + 2c_{id} w_d.
				\end{aligned}
			\]
			We obtain
			\begin{equation}  \label{eq:sum-cij-zero}
					0   = 	
					\sum_{j \ne i} c_{ij} \left( \frac{v_i w_j - w_i v_j}{v_i} +  3 \mu_i v_j\right) 
					+ c_{id} (2w_d - 3v_i^2)   +  \sum_{j \ne i} \frac{v_j}{\mu_i - \mu_j} (w_i c_{jd} - w_j c_{id}).
			\end{equation}

			Denote
			\[
				f_{ij} = \frac{v_i w_j - w_i v_j}{v_i} +  3 \mu_i v_j, 
			\]
			we plug the values of $c_{ij}$ from \eqref{eq:cij} into \eqref{eq:sum-cij-zero} to get
			\begin{equation}  \label{eq:final-eq-long}
				\begin{aligned}
					0 &=  
					\sum_{j \ne i} \frac{f_{ij}}{\mu_i + \mu_j}(c_{id}v_j + c_{jd} v_i)
					+ \sum_{j \ne i} \frac{2 f_{ij}}{3(\mu_i^2 - \mu_j^2)} (w_i c_{jd} - w_j c_{id})  \\
						& \quad  +  \sum_{j \ne i} \frac{v_j}{\mu_i - \mu_j} (w_i c_{jd} - w_j c_{id})
					 + c_{id}(2w_d - 3 v_i^2 )  \\
						& = 
						\sum_{j \ne i} \frac{f_{ij}}{\mu_i + \mu_j}(c_{id}v_j + c_{jd} v_i)
						+ \sum_{j \ne i} \frac{g_{ij}}{\mu_i^2 - \mu_j^2} (w_i c_{jd} - w_j c_{id}) 
						+ c_{id}(2w_d - 3 v_i^2)  ,
				\end{aligned}  
			\end{equation}
			where
			\[
				g_{ij} =  \tfrac23 f_{ij} + v_j(\mu_i + \mu_j). 
			\]

			Rewrite \eqref{eq:final-eq-long} as
			\[
					0 = 
					\sum_{j \ne i} \left( \frac{f_{ij}}{\mu_i + \mu_j} v_i + \frac{g_{ij}}{\mu_i^2 - \mu_j^2} w_i \right) c_{jd}
					+
					\sum_{j \ne i} \left( \frac{f_{ij}}{\mu_i + \mu_j} v_i - \frac{g_{ij}}{\mu_i^2 - \mu_j^2} w_j  - 3 v_i^2 \right)  c_{id},
			\]
			then $0 = \sum_{j = 1}^{d-1} m_{ij} c_{jd}$ where $m_{ij}$ are defined by
			\begin{equation}  \label{eq:mij}
				\begin{aligned}
					m_{ij} 
& = \frac{f_{ij}}{\mu_i + \mu_j} v_i + \frac{g_{ij}}{\mu_i^2 - \mu_j^2} w_i ,
&  1 \le i \ne j \le d-1, \\
m_{ii} 
& = \sum_{j \ne i} \frac{f_{ij}}{\mu_i + \mu_j} v_j 
- \sum_{j \ne i} \frac{g_{ij}}{\mu_i^2 - \mu_j^2} w_j
					 + 2w_d - 3 v_i^2, & 1 \le i \le d-1.
				\end{aligned}
			\end{equation}

			If the matrix $M = M(v, \mu) =  (m_{ij})$ is invertible for $\mu \in \R^{d-1}$, then $c_{id} = 0$ for all $i = 1, \cdots, d-1$. Substituting back, we see that $c_{ij} = b_{ij} = 0$ for $1 \le i \le j \le d-1$, which then imply all coefficients vanish.

			We now for the first time invoke the dependence of $v$ and $w$ on the parameters, namely
			\[
					G = \bmat{\bar{G} & 0 \\ 0 & 1}, \quad v = v(G) =  \bar{G} \bar{v}, \quad w(G, \mu) = - \bmat{\bar{G} \bar{w} \\ \hat{w}} - 2 \bmat{\Gamma v(G) \\ 0},
			\]
			where $\Gamma = \diag\{\mu_1, \cdots, \mu_{d-1}\}$.

			Let us first deal with the case $d = 2$. In this case the parameter $\bar{G}$ is always the $1 \times 1$ identity matrix, and the vector $\bar{v} = v_1$ is one-dimensional and nonzero by assumption. Moreover, $w_d = -\hat{w} =  v_1^2$ (see \eqref{eq:bn-notation}). It follows that $\det M(v(G), \mu) = 2w_d - 3 v_1^2 =  - v_1^2 \ne 0$.

			Suppose now $d \ge 3$. Since $\bar{v} \ne 0$, for an open and dense set $\cV_1$ of orthogonal matrices, for all $G \in \cV_1$, $G\bar{v}$ satisfies \eqref{eq:vi}. For a fixed $G \in \cV_1$, it suffices to show that $\det M(v(G), \mu)$ is not a trivial function. To do this, we consider the function $\det M(v(G), t\mu)$ where $t \gg 1$.

			Since
			\[
					w(G, \mu) = - \bmat{G\bar{w} \\ \hat{w}} - 2 \bmat{\mu_1 v_1 & \cdots & \mu_{d-1} v_{d-1} & 0}^T.
			\]
			then as $t \to \infty$, 
			\[
				w(t \mu) = - 2 t \bmat{\mu_1 v_1 & \cdots & \mu_{d-1} v_{d-1} & 0}^T + O(1),
			\]
			where $O(1)$ is a bounded term as $t \to \infty$ (since $G$ is orthogonal, $G\bar{w}$ is uniformly bounded). We get $w_j(t \mu) = t \mu_j v_j + O(1)$, hence
			\[
				f_{ij}(t \mu) =  2 t\frac{v_i  \mu_j v_j - v_j  \mu_i v_i}{v_i} + 3  t\mu_i v_j + O(1)
				=   t v_j( \mu_i + 2\mu_j) + O(1),
			\]
			\[
				g_{ij}(t\mu) = 
				  \tfrac23   t v_j(\mu_i + 2\mu_j) + t v_j(\mu_i + \mu_j) + O(1) 
				=  tv_j(\tfrac53 \mu_i + \tfrac73 \mu_j) + O(1).
			\]

			Plug into \eqref{eq:mij}, we get
			\begin{equation}  \label{eq:mij-large-t}
				\begin{aligned}
					m_{ij}(v, t\mu) 
& = \bar{m}_{ij}(v, \mu) + O(1/t) \\ 
& = v_i v_j \frac{8 \mu_i^2 + 10 \mu_i \mu_j - 6 \mu_j^2}{3(\mu_i^2 - \mu_j^2)} + O(1/t)
&  1 \le i \ne j \le d-1, \\
m_{ii}(v, t\mu) 
& = \bar{m}_{ij}(v, \mu) + O(1/t) \\
& =  \sum_{j \ne i} v_j^2 \frac{3 \mu_i^2 - 2 \mu_i \mu_j - 13\mu_j^2}{3(\mu_i^2 - \mu_j^2)} 
					 - 3 v_i^2 + O(1/t), & 1 \le i \le d-1.
				\end{aligned}
			\end{equation}
			If $\det \bar{M}(v(G), \mu) = \det (\bar{m}_{ij})(v(G), \mu) \ne 0$, then $\det M(v(G), t\mu) \ne 0$ for sufficiently large $t$. 

			We will first show that $\det \bar{M}(v, \mu)$ is a nontrivial rational function of $v$, $\mu$. Indeed, we will set $v_i = 0$ for all $3 \le i \le d-1$, and check that $\det \bar{M}$ is still nontrivial. Let us denote
			\[
				p_{ij} = 8 \mu_i^2 + 10 \mu_i \mu_j - 6 \mu_j^2, \quad 
				q_{ij} = 3 \mu_i^2 - 2 \mu_i \mu_j - 13 \mu_j^2,
			\]
			then 
			\[
\begin{aligned}
		\bar{m}_{12} & = v_1 v_2 \frac{p_{12}}{3(\mu_1^2 - \mu_2^2)}, 
								 & \bar{m}_{21} 
								 & = v_1 v_2 \frac{p_{21}}{3(\mu_1^2 - \mu_2^2)},  \\
		\bar{m}_{11} & = v_2^2 \frac{q_{12}}{3(\mu_1^2 - \mu_2^2)} - 3 v_1^2, 
								 & \bar{m}_{22} 
								 & = v_1^2 \frac{q_{21}}{3(\mu_2^2 - \mu_1^2)} - 3 v_2^2, \\
		\bar{m}_{ii} & = \sum_{j = 1, 2} v_j^2 \frac{q_{ij}}{3(\mu_i^2 - \mu_j^2)}- 3 v_i^2, 
								 & & 3 \le i \le d-1, \\	
		\bar{m}_{ij} & = 0, && \text{ otherwise}.
\end{aligned}
			\]
			Hence
			\[
					\det \bar{M}(v, \mu) |_{v_3 = \cdots = v_{d-1} = 0}
					= \left( \prod_{i \notin {1, 2}} \bar{m}_{ii}\right) 
					\det \bmat{\bar{m}_{11} & \bar{m}_{12} \\ \bar{m}_{21} & \bar{m}_{22}}
			\]
			where (denote $p_{ij} $)
			\[
\begin{aligned}
			 &  \det \begin{bmatrix}
			  	\bar{m}_{11} & \bar{m}_{12} \\ \bar{m}_{21} & \bar{m}_{22}
		\end{bmatrix} \\
				& = \frac{1}{9(\mu_1^2 - \mu_2^2)^2} \det 
\begin{bmatrix}
		v_2^2 q_{12} - 9 v_1^2(\mu_1^2 - \mu_2^2) & v_1 v_2 p_{12} 	 \\
		-  v_1 v_2 p_{21}  & - v_1^2 q_{21} - 9 v_2^2(\mu_1^2 - \mu_2^2) 
\end{bmatrix} \\
				& =  \frac{1}{9(\mu_1^2 - \mu_2^2)^2} 
				\bigl(  v_1^2 v_2^2 (- q_{12} q_{21}  + p_{12} p_{21} + 81 (\mu_1^2 - \mu_2^2)^2) \\
				& \quad + 9 v_1^4 (\mu_1^2 - \mu_2^2) q_{21} - 9 v_2^4 q_{12}(\mu_1^2 - \mu_2^2)
				\bigr).
\end{aligned}
			\]
			Since each of $\bar{m}_{ii}$ for $3 \le i \le d-1$ are nontrivial functions, and so is $\det \begin{bmatrix}
			  	\bar{m}_{11} & \bar{m}_{12} \\ \bar{m}_{21} & \bar{m}_{22}
			\end{bmatrix} $ (for example, the coefficients of $v_1^4$ and $v_2^4$ are clearly nontrivial), we conclude $\det \bar{M}(v, \mu)$ is a nontrivial function of $v, \mu$.

			 We remark that the scaling limits taken to prove nontriviality of the determinant is purely an artifact of the proof and unrelated to any intrinsic property of the system. 

			Finally, we have concluded that $\det \bar{M}$ as a rational function of $v$ and $\mu$ is nontrivial, therefore nonvanishing for an open and dense set of $v, \mu \in \R^{d-1}$. Observe that $\det \bar{M}$ is a homogeneous function in $v$, therefore we can restrict to the space of $v \in \{\|v\| = \|\bar{v}\|\}, \mu \in \R^{d-1}$. Finally, since the group action by $\bar{G} \in \cO(d-1)$ is transitive on $\{\|v\| = \|\bar{v}\|\}$, the same holds for  $(v(G), \mu)$ where $(G, \mu)$ is taken over an open and dense set. 
		\end{proof}

		It remains to prove Proposition \ref{prop:B5}. 

		\begin{proof}[Proof of Proposition \ref{prop:B5}]
			The top right block of $B_{ij}^5$ is
			\[
				- 4\left( \ddot{A} E_{ij} A + A E_{ij} \ddot{A} \right)  - 6 \dot{A} E_{ij} \dot{A}.
			\]
			We now show that the $d, d$ entry of the above matrix is non-zero for some $1 \le i = j \le d-1$.
			Note that $\ddot{A} E_{ij} A + A E_{ij} \ddot{A}$ vanishes on the $d, d$ entry due to the block form of $A(0)$. Using \eqref{eq:eta-ij}, we calculate.
			\[
				\begin{aligned}
					\dot{A}(0) E_{ii} \dot{A}(0)
		& = (2 \mu_i F_{ii} - 2 v_i F_{di}) 
		\left( \sum_{k = 1}^{d-1} \mu_k F_{kk} + \sum_{k = 1}^{d-1} v_k (F_{dk} + F_{kd}) \right) \\
		& = - 2 v_i^2 F_{dd} + \cdots
				\end{aligned}
			\]
			where we have omitted all the terms independent of $F_{dd}$. Since by assumption, the vector $(v_1, \cdots, v_{d-1}) \ne 0$, the proposition follows.
		\end{proof}

		\appendix

		\section{Uniqueness of normal lifts}
		\label{sec:normal-lifts}

		In the classical Lagrangian mechanics, the Legendre transform 
		\[
				\bL(q, v) = (q, \partial_v L(q, v))  , \quad
				\bL^{-1}(q, p) = (q, \partial_p H(q, p))
		\]
		is a diffeomorphism between $TM$ and $T^*M$. In the sub-Riemannian setting, $\bL$ is not defined on all of $TM$, while $\bL^{-1}$ is not one-to-one. It is therefore possible that two different normal orbit segments $(Q_1, P_1)$, $(Q_2, P_2)$ will project to the same curve in $M$:
            \begin{equation}\label{same_lifts_eq}
                 Q_1 = Q_2, \quad P_1 \ne P_2.  
            \end{equation}
The main goal of this appendix is to show that if $\theta|_{[t_0,t_1]}$  is a $\cD$-regular segment, then $\theta|_{[t_0,t_1]}$ is the unique lift of $\pi \circ \theta|_{[t_0,t_1]}$. This is needed in the proof of Proposition \ref{tagged_interval}. 

This claim is included as an exercise for the sub-Riemannian case in \cite{Vit14}, here we include a proof for the general subbundle Lagrangian case.

We consider a local frame of $\cD$
\[
  f_1, \cdots, f_m  
\]
defined on an open set $U \subset M$. There exists a smooth function $\varphi(q, \c)$ defined on $U \times \R^m$ such that
\[
  L(q, \sum_i \c_i f_i(q)) = \varphi(q, \c).  
\]

The following theorem is standard in sub-Riemannian geometry (\cite{Montgomerybook, Rif14, Agrachevbook}). For the $\cD$-Hamiltonian version, we refer to \cite{AS04}, Theorem 12.10.

\begin{theorem}\label{thm:pontryagin}
		Let $Q: [0, \delta] \to U$ be a horizontal curve and 
		\[
				\dot{Q}(t) = \sum_i \c_i f_i\big(Q(t)\big).
		\]
		\begin{enumerate}[(a)]
 \item If $Q$ lifts to a normal curve $\big(Q(t), P(t)\big)$, then
\begin{equation}  \label{eq:normal}
				 \dot{P} = -  P \cdot \sum_i \c_i \partial_q f_i(Q) + \partial_q \varphi(Q, \c),
\end{equation}
and 
\begin{equation}  \label{eq:energy}
P \cdot \sum_i \c_i f_i(Q) - \varphi(Q, \c) = \max_{v\in T_QU} \left( P \cdot \sum_i v_i f_i(Q) - \varphi(Q, v) \right) = H(Q, P). 
\end{equation}
\item $Q$ is singular if and only if there exists a nonzero $\eta(t) \in T^*_{Q(t)}M$ such that
\begin{equation}  \label{eq:singular}
		    	 \dot{\eta} = -  \eta \cdot \sum_i \c_i \partial_q f_i(Q), \quad
					 \eta(t) \cdot f_i\big(Q(t)\big) = 0, \, \forall 1 \le i \le m.
\end{equation}
\end{enumerate}
\end{theorem}

\begin{corollary}\label{cor:unique-lift}
		Suppose $Q: [0, \delta] \to U$ is a regular curve, then $Q$ admits a unique normal lift.	
\end{corollary}
\begin{proof}
		Suppose $\big(Q(t), P_1(t)\big)$ and $\big(Q(t), P_2(t)\big)$ are two normal lifts, we plug both into \eqref{eq:normal} and subtract to get
		\[
				\dot{P}_1 - \dot{P}_2
				= - (P_1 - P_2) \cdot \sum_i \c_i \partial_q f_i(Q).
		\]
		Moreover, \eqref{eq:energy} implies
		\[
		  \partial_p H(Q, P_1) = \partial_p H(Q, P_2).
		\]
		Since $H(q, \cdot)$ descends to a strictly convex function on $T^*_q M / \cD^\perp(q)$, $P_1 - P_2 \in \cD^\perp(Q)$, or $(P_1 - P_2) \cdot f_i(Q) = 0$ for all $1 \le i \le m$. Hence $\eta = P_1 - P_2$ satisfies \eqref{eq:singular}.  Since $Q$ is regular, by item (b) of Theorem \ref{thm:pontryagin}, $P_1 - P_2 = 0$.
\end{proof}

		\printbibliography

\end{document}